\documentclass{article} % For LaTeX2e
\usepackage{iclr2026_conference,times}

\usepackage{amsmath,amsfonts,bm,amsthm,amssymb}

\def\eqref#1{equation~\ref{#1}}
\def\1{\bm{1}}

\def\eps{{\epsilon}}

\def\vzero{{\bm{0}}}
\def\vone{{\bm{1}}}

\def\vu{{\bm{u}}}

\def\vw{{\bm{w}}}
\def\vx{{\bm{x}}}
\def\vy{{\bm{y}}}
\def\vz{{\bm{z}}}

\def\mA{{\bm{A}}}

\def\mH{{\bm{H}}}

\def\mP{{\bm{P}}}

\def\mU{{\bm{U}}}

\def\mW{{\bm{W}}}

\DeclareMathAlphabet{\mathsfit}{\encodingdefault}{\sfdefault}{m}{sl}
\SetMathAlphabet{\mathsfit}{bold}{\encodingdefault}{\sfdefault}{bx}{n}

\def\gA{{\mathcal{A}}}
\def\gB{{\mathcal{B}}}

\def\gF{{\mathcal{F}}}

\def\gO{{\mathcal{O}}}
\def\gP{{\mathcal{P}}}

\def\gR{{\mathcal{R}}}

\def\gT{{\mathcal{T}}}

\def\gX{{\mathcal{X}}}
\def\gY{{\mathcal{Y}}}
\def\gZ{{\mathcal{Z}}}

\def\sN{{\mathbb{N}}}
\def\sO{{\mathbb{O}}}

\def\sR{{\mathbb{R}}}

\newcommand{\E}{\mathbb{E}}

\newtheorem{thm}{Theorem}[section]
\newtheorem{dfn}{Definition}[section]
\newtheorem{exmp}{Example}[section]
\newtheorem{lem}{Lemma}[section]
\newtheorem{asm}{Assumption}[section]
\newtheorem{rmk}{Remark}[section]

\usepackage[table]{xcolor}
\usepackage{hyperref}
\usepackage{url}
\usepackage{thmtools}
\usepackage{thm-restate}
\usepackage{graphicx}
\usepackage{booktabs} 
\usepackage{multirow}
\usepackage{algorithm,algorithmic}
\usepackage{xcolor}
\usepackage{nicefrac}

\definecolor{darkgreen}{RGB}{34, 139, 34} 
\definecolor{darkblue}{RGB}{30, 60, 150}  
\title{Faster Gradient Methods for  Highly-smooth Stochastic Bilevel Optimization}

\author{Lesi Chen and Junru Li\\
Tsinghua University \& Shanghai Qizhi Institude, China \\
\texttt{\{chenlc23,jr-li24\}@mails.tsinghua.edu.cn} \\
\AND
El Mahdi Chayti \\
École Polytechnique Fédérale de Lausanne (EPFL), Switzerland \\
\texttt{el-mahdi.chayti@epfl.ch} \\
\AND
Jingzhao Zhang $^\dagger$  \\
Tsinghua University \& Shanghai Qizhi Institude,China \\
\texttt{jingzhaoz@mail.tsinghua.edu.cn}
}

\iclrfinalcopy % Uncomment for camera-ready version, but NOT for submission.
\begin{document}

\maketitle

\begin{abstract}
This paper studies the complexity of finding an $\epsilon$-stationary point for stochastic bilevel optimization when the upper-level problem is nonconvex and the lower-level problem is strongly convex.
Recent work proposed the first-order method, F${}^2$SA, achieving the $\tilde \gO(\epsilon^{-6})$ upper complexity bound for first-order smooth problems. This is slower than the optimal $\Omega(\epsilon^{-4})$ complexity lower bound in its single-level counterpart. 
In this work, we show that faster rates are achievable for higher-order smooth problems. We first reformulate F$^2$SA as approximating the hyper-gradient with a forward difference. Based on this observation, we propose a class of methods F${}^2$SA-$p$ that uses $p$th-order finite difference for hyper-gradient approximation and improves the upper bound to $\tilde \gO(p \epsilon^{-4-2/p})$ for $p$th-order smooth problems. Finally, we demonstrate that the
$\Omega(\epsilon^{-4})$ lower bound
also holds for stochastic bilevel problems when the high-order smoothness holds for the lower-level variable, indicating that the upper bound of F${}^2$SA-$p$ is nearly optimal in the region $p = \Omega(\log \epsilon^{-1} / \log \log \eps^{-1})$. 
%In this work, we show that small modifications to F${}^2$SA can achieve faster rates for highly-smooth bilevel problems.
%which still has a gap between the $\Omega(\epsilon^{-4})$ lower bound. 
%  we then propose F$^2$SA-2 that uses the central difference for the hyper-gradient approximation, achieving the improved  $\tilde \gO(\epsilon^{-5})$ upper bound for second-order smooth problems.
% Further, we show that F${}^2$SA-2 falls in a more general algorithm class F${}^2$SA-$p$, which can achieve the
% $\tilde \gO(p \epsilon^{-4-p/2})$ upper bound for $p$th-order smooth problems. 
\end{abstract}
{
\renewcommand\thefootnote{$^\dagger$}
\footnotetext{The corresponding author.}}
\section{Introduction}

Many machine learning problems, such as meta-learning \citep{rajeswaran2019meta}, hyper-parameter tuning \citep{bao2021stability,franceschi2018bilevel,mackay2019self}, adversarial training \citep{goodfellow2020generative}, and reinforcement learning \citep{yang2019provably,hong2023two,zeng2024two,zeng2024fast} can be abstracted as solving the following bilevel optimization problem:
\begin{align} \label{prob:main-BLO}
    \min_{\vx \in \sR^{d_x}} \varphi(\vx) = f(\vx,\vy^*(\vx)) , \quad \vy^*(\vx) = \arg \min_{\vy \in \sR^{d_y}} g(\vx,\vy),
\end{align} We call $f$ and $g$ the upper-level and lower-level functions, respectively, and call $\varphi$ the \textit{hyper-objective}. In this paper, we consider the most common \textit{nonconvex-strongly-convex} setting where $f:\sR^{d_x} \rightarrow \sR $ is smooth and possibly nonconvex, and $g: \sR^{d_y} \rightarrow \sR$ is smooth jointly in $(\vx,\vy)$ and strongly convex in $\vy$. Under the lower-level strong convexity assumption, the implicit function theorem \citep{dontchev2009implicit} indicates the following closed form of the hyper-gradient
\begin{align} \label{eq:hyper-grad}
    \nabla \varphi(\vx) = \nabla_x f(\vx,\vy^*(\vx)) - \nabla_{xy}^2 g(\vx,\vy^*(\vx)) [ \nabla_{yy}^2 g(\vx,\vy^*(\vx))]^{-1} \nabla_y f(\vx,\vy^*(\vx)).
\end{align} 
Following the works in nonconvex optimization \citep{carmon2020lower,carmon2021lower,arjevani2023lower}, we consider the task of finding an $\epsilon$-stationary point of $\varphi$, \textit{i.e.}, a point $\vx \in \sR^{d_x}$ such that $\Vert \nabla \varphi(\vx) \Vert \le~\epsilon$.
Motivated by many real machine learning tasks, we study the stochastic setting, where the algorithms only have access to the stochastic derivative estimators of both $f$ and $g$.

The first efficient algorithm BSA \citet{ghadimi2018approximation} for solving the stochastic bilevel problem leverages both stochastic gradient and Hessian-vector-product (HVP) oracles to find an $\epsilon$-stationary point of $\varphi(\vx)$.
Subsequently, \citet{ji2021bilevel} proposed stocBiO by incorporating multiple enhanced designs in BSA to improve the complexity. % Both BSA and  
% stocBiO can be implemented with stochastic Hessian-vector product oracles, which are almost as efficient as stochastic gradients.
\citet{yang2023achieving} proposes FdeHBO that uses finite-differences to estimate HVP vectors. However, 
all these methods require the stochastic Hessian assumption~(\ref{eq:asm-stocBio}) on the lower-level function, which is stronger than the standard SGD assumption.

% To avoid estimating HVP oracles, 

\citet{kwon2023fully} proposed the first fully first-order method F${}^2$SA that works under standard SGD assumptions on both $f$ and $g$ (Assumption \ref{asm:stoc}).
The main idea is to solve the following penalty problem \citep{liu2022bome,liu2023value,shen2023penalty,shen2025penalty,lu2024first,lu2026solving}:
\begin{align} \label{eq:penalty-Kwon}
    \min_{\vx \in \sR^{d_x}, \vy \in \sR^{d_y}} f(\vx,\vy) + \lambda \left(g(\vx,\vy) - \min_{\vz \in \sR^{d_y}} g(\vx,\vz) \right),
\end{align}
where $\lambda$ is taken to be sufficiently large such that $\lambda = \Omega(\epsilon^{-1})$. {If we 
interpret $\lambda$ as the Lagrangian multiplier, then Problem (\ref{eq:penalty-Kwon}) can be viewed as
the Lagrangian function of the constrained optimization $\min_{\vx \in \sR^{d_x}, \vy \in \sR^{d_y}} f(\vx,\vy), {\rm s.t.} ~ g(\vx,\vy) \le g(\vx,\vy^*(\vx))$.}
Thanks to Danskin's theorem, the gradient of the Problem (\ref{eq:penalty-Kwon}) only involves gradient information.
Therefore, F${}^2$SA does not require the stochastic Hessian assumptions (\ref{eq:asm-stocBio}). More importantly, by directly leveraging gradient oracles instead of more expensive HVP oracles, the 
F${}^2$SA is more efficient in practice \citep{shen2025seal,xiao2025unlocking,jiang2025beyond} and is also the only method that can be scaled to $32$B sized large language model (LLM) training \citep{pan2024scalebio}.

\citet{kwon2023fully} proved that the F${}^2$SA method finds an $\epsilon$-stationary point of $\varphi(\vx)$ with
 $\tilde \gO(\epsilon^{-3})$ first-order oracle calls in the deterministic case and 
$\tilde \gO(\epsilon^{-7})$ stochastic first-order oracle (SFO) calls in the stochastic case. Recently, \citet{chen2023near} showed the two-time-scale stepsize strategy improves the upper complexity bound of F${}^2$SA method to $\tilde \gO(\epsilon^{-2})$ in the deterministic case, which is optimal up to logarithmic factors. However, the direct extension of their method in the stochastic case leads to the 
$\tilde \gO(\epsilon^{-6})$ SFO complexity \citep{chen2023near,kwon2024complexity} , which still has a significant gap between the $\Omega(\epsilon^{-4})$ lower bound for SGD \citep{arjevani2023lower}.
It remains open whether optimal rates for stochastic bilevel problems can be achieved for fully first-order methods.

In this work, we revisit F${}^2$SA and interpret it as using forward difference to approximate the hyper-gradient. Our novel interpretation in turn leads to straightforward algorithm extensions for the F${}^2$SA method. Observing that the forward difference used by F${}^2$SA only has a first-order error guarantee, a natural idea to improve the error guarantee is to use higher-order finite difference methods.
For instance, we know that the central difference has an improved second-order error guarantee. Based on this fact, we can derive the F${}^2$SA-2 method that solves the following symmetric penalty problem:
\begin{align} \label{eq:penalty-F2SA-2}
    \min_{\vx \in \sR^{d_x}, \vy \in \sR^{d_y}} \frac{1}{2} \left( f(\vx,\vy) + \lambda g(\vx,\vy) - \min_{\vz \in \sR^{d_y}} \left( - f( \vx,\vz) + \lambda g(\vx,\vz) \right) \right).
\end{align}
{Compared with Eq. (\ref{eq:penalty-Kwon}), this new penalty problem perturbs the lower-level variables $\vy$ and $\vz$ in the opposite direction to better cancel out the approximation errors to Problem (\ref{prob:main-BLO}).}
The connection between bilevel optimization and finite difference approximations was recently established by \citet{chayti2024new} in the context of meta-learning, but their results were limited to symmetric approximations. We extend their findings beyond meta-learning and to general finite difference approximations, addressing their conjecture about broader applicability. In this work, we show that F${}^2$SA-2 provably improves the SFO complexity of F${}^2$SA from $\tilde \gO(\epsilon^{-6})$ to $\tilde \gO(\epsilon^{-5})$ for second-order smooth problems. Moreover, our idea is generalizable for any $p$th-order smooth problems. It is known in numerical analysis there exists the $p$th-order central difference that uses $p$ points to construct an estimator to the derivative of a unitary function with $p$th-order error guarantee, as recalled in Lemma~\ref{lem:error-pth}. Motivated by this fact, we propose the F${}^2$SA-$p$ algorithm 
%which solves the penalty problem
% \begin{align*}
% \min_{\vx \in \sR^{d_x}}  \sum_{j=-p/2}^{j = p/2} \alpha_j \left( \min_{\vy \in \sR^{d_y}}  j f(\vx,\vy) + \lambda g(\vx,\vy) \right),
% \end{align*}
and show that it enjoys the improved $\tilde \gO(p \epsilon^{-4-2/p})$ SFO complexity, as formally stated in Theorem \ref{thm:upper-pth}.

Moreover, as formally stated in Theorem \ref{thm:lower}, it is easy to extend the $\Omega(\epsilon^{-4})$ lower bound for SGD \citep{arjevani2023lower} from single-level optimization to bilevel optimization using a fully separable construction that automatically satisfies all our additional smoothness conditions in Definition \ref{dfn:highly-smooth-blo}. It shows that F${}^2$SA-$p$ is optimal up to logarithmic factors when $p = \Omega(\log \epsilon^{-1} / \log \log \epsilon^{-1})$ (see Remark \ref{rmk:high-p-compl}). 
We summarize our main results in Table~\ref{tab:main-res} and discuss open problems in the following.

% To examine the tightness of our upper bounds,
% we further extend the . Note that existing constructions for 
% bilevel lower bound \citep{dagreou2024lower,kwon2024complexity} do not satisfy . We demonstrate in Theorem~\ref{thm:lower} that a  can immediately yield a valid $\Omega(\epsilon^{-4})$ lower bound for the problem class we study, showing that 

%For very large $p$, our upper bounds can approach the standard $\Omega(\epsilon^{-4})$ lower bound for SGD \citep{arjevani2023lower}.
% However .....which we prove still holds under our highly smooth assumptions 

% \paragraph{Our contributions.}
%  with the following results:
% \begin{itemize}
%     \item We propose F${}^2$SA-2 method, which makes minimal changes to F${}^2$SA, but improves its complexity bound from $\tilde \gO(\epsilon^{-5})$ to $\tilde \gO(\epsilon^{-4})$ under second-order smoothness assumptions.
%     \item We extend  F${}^2$SA-2 to higher-order smooth problems and propose a class of methods called F$^2$SA-$p$ that have 
%     the $\tilde \gO(\epsilon^{-4 - 2/p})$ upper bound under $p$th-order smoothness assumptions.
%     \item We prove an $\Omega(\epsilon^{-4})$ lower bound via a reduction to single-level nonconvex optimization, showing that F$^2$SA-$p$ is near-optimal for a sufficiently large constant $p$.
% \end{itemize}

\paragraph{Open problems.} 
Our upper bounds improve known results for high-order smooth problems, but our result 
still has a gap between the $\Omega(\epsilon^{-4})$ lower bound for $p = \gO(\log \epsilon^{-1} / \log \log \epsilon^{-1})$.
Recently, \citet{kwon2024complexity} obtained some preliminary results towards closing this gap for $p=1$, where they showed an $\Omega(\epsilon^{-6})$ lower bound holds under a more adversarial oracle.  But it is still open whether their lower bounds can be extended to standard stochastic oracles.
Another open problem is the tightness of the condition number dependency, for which the current upper and lower bounds have a gap of $\Omega(\kappa^9)$ as demonstrated in Table \ref{tab:main-res}. Two recent concurrent works \citep{ji2025lower,chen2025condition} proposed tighter lower bounds for $p=1$: \citet{ji2025lower} showed a lower bound of $\Omega(\kappa^{5/2} \epsilon^{-4})$ and \citet{chen2025condition} showed that of  $\Omega(\kappa^{4} \epsilon^{-4})$. However, it is still open to fully close the gap in condition number dependency for both $p=1$ and $p \ge 2$.

\begin{table}[htbp]
    \centering
    \begin{tabular}{c c c c}
    \hline 
    Method   & Smoothness  &  Reference & Complexity  \\
    \hline \hline \addlinespace
    F${}^2$SA &  
    1st-order
    &\citep{kwon2023fully} & $\tilde \gO({\rm poly}(\kappa) \epsilon^{-7})$ \\ \addlinespace
    F${}^2$SA &
    1st-order &\citep{kwon2024complexity} & $\tilde \gO({\rm poly}(\kappa) \epsilon^{-6})$ \\ \addlinespace
    F${}^2$SA    & 1st-order & \citep{chen2023near} & $\tilde \gO(\kappa^{12} \epsilon^{-6})$
    \\  \addlinespace
    % F${}^2$SA-2 & Theorem \ref{thm:upper-improved} & $\tilde \gO(\kappa^{10}\epsilon^{-5})$  \\ \addlinespace
    \vspace{-1mm}
    F${}^2$SA-$p$ & \cellcolor{gray!20} 1st-order  & Theorem \ref{thm:upper-pth} & $\tilde \gO(p \kappa^{9+2/p} \epsilon^{-4-2/p})$ \\ \vspace{-1mm}
    & \cellcolor{gray!20} + & \\
    Lower Bound & \cellcolor{gray!20} $p$th-order in $\vy$   & \citep{arjevani2023lower} & $\Omega(\epsilon^{-4})$ \\ \addlinespace
    \hline
    \end{tabular}
    \caption{The SFO complexity of different methods to find an $\epsilon$-stationary point for $p$th-order smooth first-order bilevel problems with condition number $\kappa$ under standard SGD assumptions. 
    %We remark in ``*'' that \citet{kwon2024complexity} do not explicitly give the $\kappa$ dependency in their complexity analysis.
    }
    \label{tab:main-res}
\end{table}

%Our work improves known upper bounds but a gap .....

\paragraph{Notations.}
We use $\Vert \,\cdot \, \Vert$ to denote the Euclidean norm for vectors and the spectral norm for matrices and tensors. 
For any set $S$ and functions $g, h :  S \rightarrow [0, \infty)$ we write $g = \gO(h)$ or $h = \Omega(g)$ equivalently if
there exists $c > 0$ such that $g(s) \le c  h(s)$ for every $s \in S$. 
We use $\tilde \gO(\,\cdot\,)$ and $\tilde \Omega(\,\cdot\,)$ to hide logarithmic factors in $\gO(\,\cdot\,)$ and $\Omega(\,\cdot\,)$. We  alsouse $h_1 \lesssim h_2$ to mean $h_1 = \gO(h_2)$,  
$h_1 \gtrsim h_2$ to mean $h_1 = \Omega(h_2)$, and $h_1 \asymp h_2$ to mean that both $h_1 \lesssim h_2$ and 
$h_1 \gtrsim h_2$ hold. Additional notations for tensors are introduced in Appendix~\ref{apx:notation-tensor}.

% Our method is motivated by the central  for approximating derivatives, 

% It corresponds to the forward difference formula, which has first-order approximation error.

% This work proposes

% It corresponds to the central difference formula, which has second-order approximation error under additional smoothness assumptions.

% Finally, we derive a class of fully first-order method for highly-smooth functions using $p$th-order formula of finite difference, giving the upper bound of , which approaches the $\Omega(\epsilon^{-4})$ lower bound when $p \rightarrow +\infty$.

% There are plenty of works that demonstrates the benefit of high-order smoothness in the upper-level variable $\vx$. However, to the best of our knowledge, we are the first to demonstrate the benefit of high-order smoothness in the lower-level variable $\vy$.

% \section{Related Works}

\section{Preliminaries}

The goal of bilevel optimization is to minimize the hyper-objective $\varphi(\vx)$, which is in general nonconvex. Since finding a global minimizer of a general nonconvex function requires exponential complexity in the worst case \citep[§ 1.6]{nemirovskij1983problem}, we follow the literature \citep{carmon2020lower,carmon2021lower} to consider the task of finding an approximate stationary point.

{\begin{dfn}
Let $\varphi : \sR^{d_x} \rightarrow \sR$ be the hyper-objective defined in Eq. (\ref{prob:main-BLO}). We say a random variable $\hat \vx \in \sR^{d_x}$ is an $\epsilon$-hyper-stationary point if $\E \Vert \nabla \varphi(\hat \vx) \Vert \le \epsilon$.
\end{dfn}}

Next, we introduce the assumptions used in this paper, which ensure the tractability of the above hyper-stationarity. Compared to \citep{kwon2023fully,chen2023near}, we additionally assume the high-order smoothness in lower-level variable $\vy$ to achieve further acceleration.

\subsection{Problem Setup}
First of all, we follow the standard assumptions on SGD \citep{arjevani2023lower} to assume that the stochastic gradient estimators satisfy the following assumption.
\begin{asm} \label{asm:stoc}
There exists stochastic gradient estimators $F(\vx,\vy)$ and $G(\vx,\vy)$ such that
\begin{align*}
    \E F(\vx,\vy; \xi) = \nabla f(\vx,\vy), \quad \E \Vert F(\vx,\vy) - \nabla f(\vx,\vy) \Vert^2 \le \sigma^2; \\
    \E G(\vx,\vy; \zeta) = \nabla g(\vx,\vy), \quad \E \Vert G(\vx,\vy) - \nabla g(\vx,\vy) \Vert^2 \le \sigma^2,
\end{align*}
where $\sigma>0$ is the variance of the stochastic gradient estimators.
We also partition $F = (F_x,F_y)$ and $G = (G_x,G_y)$ such that $F_x,F_y,G_x,G_y$ are estimators to $\nabla_x f, \nabla_y f, \nabla_x g, \nabla_y g$, respectively.
\end{asm}

Second, we assume that the hyper-objective $\varphi(\vx) = f(\vx,\vy^*(\vx))$ is lower bounded. Otherwise, any algorithm requires infinite time to find a stationary point. Note that the implicit condition $\inf_{\vx \in \sR^{d_x}} \varphi(\vx) > - \infty$ below can also be easily implied by a more explicit condition 
on the lower boundedness of upper-level function $\inf_{\vx \in \sR^{d_x}, \vy \in \sR^{d_y}} f(\vx,\vy) > - \infty$.

\begin{asm} \label{asm:phi-lb}
The hyper-objective defined in Eq. (\ref{prob:main-BLO}) is lower bounded, and we have
    \begin{align*}
        \varphi(\vx_0) - \inf_{\vx \in \sR^{d_x}} \varphi(\vx) \le \Delta,
    \end{align*}
where $\Delta>0$ is the initial suboptimality gap and we assume $\vx_0 = \vzero$ without loss of generality. 
\end{asm}

Third, we assume the lower-level function $g(\vx,\vy)$ is strongly convex in $\vy$.
It guarantees the uniqueness of $\vy^*(\vx)$ and the tractability of the bilevel problem. Although not all the problems in applications satisfy the lower-level strong convexity assumption, it is impossible to derive dimension-free upper bounds when the lower-level problem is only convex \citep[Theorem 3.2]{chen2024finding}. Hence, we follow most existing works to consider strongly convex lower-level problems.

\begin{asm} \label{asm:LLSC}
    $g(\vx,\vy)$ is $\mu$-strongly convex in $\vy$, \textit{i.e.}, for any
    $\vy_1,\vy_2 \in \sR^{d_y}$, we have
    \begin{align*}
        g(\vx,\vy_2) \ge g(\vx,\vy_1) + \langle \nabla_y g(\vx,\vy_1), \vy_2 - \vy_1 \rangle + \frac{\mu}{2} \Vert \vy_1 - \vy_2 \Vert^2,
    \end{align*}
    where $\mu>0$ is the strongly convex parameter. 
\end{asm}

Fourth, we require the following smoothness assumptions following \citep{ghadimi2018approximation}. 
According to Eq. (\ref{eq:hyper-grad}), these conditions are necessary and sufficient to guarantee the Lipschitz continuity of $\nabla \varphi(\vx)$,  which further ensure the tractability of an approximate stationary point of the nonconvex hyper-objective $\varphi(\vx)$ \citep{zhang2020complexity,kornowski2022oracle}.
% on the function values or derivatives of the upper- and lower-level functions to guarantee their smoothness.
% Recall that we say a mapping $\gT: \sR^m \rightarrow \sR^{n_1 \times \cdots \times n_q} $ is $L$-Lipschitz continuous if for any $\vz,\vz' \in \sR^n$ we have $ \Vert \gT(\vz) - \gT(\vz') \Vert \le L \Vert \vz - \vz' \Vert$, where $L>0$ is the Lipschitz continuity coefficient. When $\gT$ is differentiable, assuming that $\gT$ is $L$-Lipschitz continuous is equivalent to assuming that its derivative is bounded above by $L$, \textit{i.e.}, for any $\vz \in \sR^n$ it holds that $\Vert \nabla \gT(\vz) \Vert \le L$. Now, we make the following 

% which is necessary to guarantee the smoothness of $\nabla \varphi(\vx)$ due to Eq. (\ref{eq:hyper-grad}).

\begin{asm} \label{asm:prior}
For the upper-lower function $f$ and lower-level function $g$,
we assume that 
\begin{enumerate}
     \item $f(\vx,\vy)$ is $L_0$-Lipschitz in $\vy$.
    \item $\nabla f(\vx,\vy)$ and $\nabla g(\vx,\vy)$ are $L_1$-Lipschitz jointly in $(\vx,\vy)$.
    \item $\nabla_{xy}^2 g(\vx,\vy)$ and $\nabla_{yy}^2 g (\vx,\vy)$ are $L_2$-Lipschitz jointly in $(\vx,\vy)$.
\end{enumerate}
\end{asm}

We refer to the problem class that jointly satisfies all the above 
Assumption \ref{asm:stoc}, \ref{asm:phi-lb}, \ref{asm:LLSC} and \ref{asm:prior} as first-order smooth bilevel problems, for which 
\citep{kwon2024complexity,chen2023near} showed the F${}^2$SA method achieves the $\tilde \gO(\epsilon^{-6})$ upper complexity bound. In this work, we show an improved bound under the following additional higher-order smoothness assumption on lower-level variable~$\vy$.
% In order to obtain an improved result, we introduce the following additional higher-order smoothness assumptions on the lower-level variable $\vy$.
\begin{asm}[High order smoothness in $\vy$]  \label{asm:addition}
Given $p \in \sN_+$, we assume that
\begin{enumerate}
    \item $\frac{\partial^q}{\partial \vy^q} \nabla f(\vx,\vy)$ is $L_{q+1}$-Lipschitz for all $q = 1,\cdots,p-1$.
    \item $\frac{\partial^{q+1}}{\partial \vy^{q+1}} \nabla g(\vx,\vy)$ is $L_{q+2}$-Lipschitz in $\vy$ for all $q = 1,\cdots,p-1$.
\end{enumerate}
\end{asm}

% We refer to the problems in the previous work as first-order smooth problems, while refer to the problems satisfying the additional smoothness Assumption \ref{asm:addition} as $p$th-order 
% In the following, we define the condition number for second-order smooth problems.

We refer to problems jointly satisfying all the above assumptions as $p$th-order smooth bilevel problems, and also formally define their condition numbers as follows.

\begin{dfn}[$p$th-order smooth bilevel problems] \label{dfn:highly-smooth-blo}
Given $p \in \sN_+$, $\Delta>0$, $L_0, L_1,\cdots,L_{p+1} >0$, and $\mu \le L_1$, we use $\gF^{\text{nc-sc}}(L_0,\cdots,L_{p+1},\mu, \Delta)$ to denote the set of all bilevel instances satisfying Assumption \ref{asm:phi-lb}, \ref{asm:LLSC}, \ref{asm:prior} 
and \ref{asm:addition}. For this problem class, we define the largest smoothness constant $\bar L = \max_{0 \le j \le p} L_j$ and condition number $\kappa = \bar L / \mu$.
%and satisfies the following high-order smoothness conditions:\begin{enumerate}
% \item $\nabla^q \nabla_y f(\vx,\vy)$ is $L_{q+1}$-Lipschitz jointly in $(\vx,\vy)$ for all $q= 1,\cdots,p-1$,
%     \item $\nabla^{q+1} \nabla_y g(\vx,\vy)$ is $L_q$-Lipschitz jointly in $(\vx,\vy)$ for all $q= 1,\cdots,p-1$.
% \end{enumerate}
% We define 
\end{dfn}

All our above assumptions align with \citep{chen2023near} except for the additional Assumption \ref{asm:addition}.
{A classic example of a highly smooth function is the softmax function
\citep[Lemma 2(3)]{garg2021near}. Therefore, many hyper-parameter tuning problems for logistic regression are provably highly smooth, such that our theory can be applied. We give two examples from \citep{pedregosa2016hyperparameter}: the first one aims to learn the optimal weights of each sample in a corrupted training set, and the second one aims to learn the optimal regularizer of each parameter.

\begin{exmp}[Data hyper-cleaning]
Let $\vx \in \sR^{n}$ parameterize the per-sample weight of a training set with $n$ samples via $\sigma(x_i) = \exp(x_i)/ \sum_{i=1}^n \exp(x_i)$  and $\vy \in \sR^{d}$ be the parameters of a linear model.
Let $\ell_{\rm val}$ be the logistic loss of the linear model on the validation set and $\ell_{\rm tr}^i$ be the logistic loss on the training sample $i$. The problem aims to solve
\begin{align*}
    \min_{\vx \in \sR^{n}}\ell_{\rm val}(\vy), \quad {\rm s.t.} \quad \vy \in \arg \min_{\vy \in \sR^{d}}  \sum_{i=1}^n \sigma(x_i) \ell_{\rm tr}^i(\vy).
\end{align*}
\end{exmp}

\begin{exmp}[Learn-to-regularize] \label{exmp:l2reg}
Let $\vx \in \sR^d$ parameterize the regularization matrix via $\mW_\vx = {\rm diag} (\exp(\vx))$, and $\vy \in \sR^d$ be the parameters of a linear model.
Let $\ell_{\rm val}$ and $\ell_{\rm tr}$ be the logistic loss of the linear model on the validation set and training set, respectively.
The problem  aims to solve
 \begin{align*}
    \min_{\vx \in \sR^{d}} \ell_{\rm val}(\vy), \quad {\rm s.t.} \quad \vy \in \arg \min_{\vy \in \sR^{d}} \ell_{\rm tr}(\vy) + \vy^\top \mW_{\vx} \vy.
\end{align*}
\end{exmp}}

% \begin{dfn}\label{dfn:L-kappa}
% Under Assumption \ref{asm:LLSC}, \ref{asm:prior} and \ref{asm:addition}, we define the largest smoothness constant as $\bar L = \max\{L_0,L_1,L_2,L_3 \}$ and the condition number as $\kappa= \bar L / \mu$.
% \end{dfn}

\subsection{Comparison to Previous Works}

Before we show our improved upper bound, we first give a detailed discussion on other assumptions made in related works.

% to see our differences 
% Since we are not the first work to demonstrate that additional assumptions can lead to acceleration in bilevel optimization, 

% Before showing our improved upper bound for highly smooth stochastic bilevel problems, we first give a detailed comparison of our additional assumptions and those made in related works.

\vspace{-2mm}
\paragraph{Stochastic Hessian assumption.}
\citet{ghadimi2018approximation,ji2021bilevel}  assumes the access to a stochastic Hessian estimator $\mH(\vx,\vy)$ such that 
\begin{align} \label{eq:asm-stocBio}
\E \mH(\vx,\vy) = \nabla^2 g(\vx,\vy), \quad \E \Vert \mH(\vx,\vy) - \nabla^2 g(\vx,\vy) \Vert \le \sigma^2.
\end{align}
According to \citep[Observation 1 and 2]{arjevani2020second}, such an assumption is stronger than standard SGD assumptions and equivalent to the mean-squared-smoothness assumption~(\ref{eq:asm-vr}) on the lower-level gradient estimator $G$ under the mild condition of $\nabla G(\vx,\vy) = \mH(\vx,\vy)$.
Under this assumption, in conjunction with Assumption \ref{asm:phi-lb}, \ref{asm:LLSC}, and \ref{asm:prior}, \citet{ghadimi2018approximation} proposed the
BSA method that can 
find an $\epsilon$ stationary point of $\varphi(\vx)$ with $\tilde \gO(\epsilon^{-6})$ SFO complexity and
$\tilde \gO(\epsilon^{-4})$ stochastic HVP complexity.
Later,
\citet{ji2021bilevel} further improved
the SFO complexity term to $\tilde \gO(\epsilon^{-4})$. Compared to them, we consider the setting where the algorithms only have access to stochastic gradient estimators, and make no assumptions on the stochastic Hessians.

\vspace{-2mm}
\paragraph{Mean-squared smoothness assumption.}
Besides Assumption  \ref{asm:stoc},  \ref{asm:phi-lb}, \ref{asm:LLSC},  \ref{asm:prior} and the stochastic Hessian assumption~(\ref{eq:asm-stocBio}),
\citet{khanduri2021near,yang2021provably,yang2023achieving} further assumes that the stochastic estimators to gradients and Hessians are mean-squared smooth:
\begin{align} \label{eq:asm-vr}
\begin{split}
     \E \Vert F(\vx,\vy) - F(\vx',\vy') \Vert^2 \le \bar L_1^2 \Vert (\vx,\vy) - (\vx',\vy') \Vert^2, \\
    \E \Vert G(\vx,\vy) - G(\vx',\vy') \Vert^2 \le \bar L_1^2 \Vert (\vx,\vy) - (\vx',\vy') \Vert^2, \\
    \E \Vert \mH(\vx,\vy) - \mH(\vx',\vy') \Vert^2 \le \bar L_2^2 \Vert (\vx,\vy) - (\vx',\vy') \Vert^2.
\end{split}
\end{align}
Under this additional assumption, they proposed faster stochastic methods with upper complexity bound of $\tilde \gO(\epsilon^{-3})$ via variance reduction \citep{fang2018spider,cutkosky2019momentum}. In this paper, we only consider the setting without mean-squared smoothness assumptions and study a different acceleration mechanism from variance reduction.

% However, variance reduction are typically ineffective in practice \citep{defazio2019ineffectiveness} since the 
% mean-squared smoothness constants $\bar L_1$ and $\bar L_2$ can be arbitrarily worse than the smoothness constants $L_1$ and $L_2$.

\vspace{-2mm}
\paragraph{Jointly high-order smoothness assumption.
} \citet{huang2025efficiently} introduced a second-order smoothness assumption similar to but stronger than Assumption \ref{asm:addition} when $p=2$. Specifically, they assumed the second-order smoothness jointly in $(\vx,\vy)$ instead of $\vy$ only: 
% but their assumption is stronger than ours.
% In Assumption~\ref{asm:addition} we only 
% assume the smoothness of higher derivatives of $\nabla_y f$ and $\nabla_y g$, while \citet{huang2025efficiently} assumes such smoothness also holds for $\nabla_x f$ and $\nabla_x g$, which means
\begin{align} \label{eq:asm-escape}
    \begin{split}
        \nabla^2 f(\vx,\vy) \text{ is } L_2\text{-Lipschitz jointly in } (\vx,\vy); \\
        \nabla^3  g(\vx,\vy) \text{ is } L_3\text{-Lipschitz jointly in } (\vx,\vy).
    \end{split}
\end{align}
The jointly second-order smoothness (\ref{eq:asm-escape}) ensures that the hyper-objective $\varphi(\vx)$ has Lipschitz continuous Hessians, which further allows the application of known techniques in minimizing second-order smooth objectives. \citet{huang2025efficiently} applied the technique from \citep{jin2017escape,jin2021nonconvex,xu2018first,allen2018neon2} to show that an HVP-based method can find a second-order stationary point in $\tilde \gO(\epsilon^{-2})$ complexity under the deterministic setting, and in $\tilde \gO(\epsilon^{-4})$ under the stochastic Hessian assumption (\ref{eq:asm-stocBio}). \citet{yang2023accelerating} applied the technique from \citep{li2023restarted} to accelerate the complexity HVP-based method to $\tilde \gO(\eps^{-1.75})$ in the deterministic setting. 
\citet{chen2023near} also
proposed a fully first-order method to achieve the same $\tilde \gO(\eps^{-1.75})$ complexity. Compared to these works, our work demonstrates a unique acceleration mechanism in stochastic bilevel optimization that only comes from the high-order smoothness in $\vy$.

\section{The F${}^2$SA-$p$ Method}

To introduce our method, we first recall the prior F${}^2$SA method \citep{kwon2023fully} and establish their relationship between 
finite difference schemes, which further motivates us to design better algorithms by using better finite difference formulas.

% In this section, we introduce our improved method for stochastic first-order bilevel optimization. In Section~\ref{subsec:relation}, we establish the relationship between fully first-order hyper-gradient approximation and . In Section~\ref{subsec:algo}, we design our algorithm inspired by the $p$th-order finite difference method. In Section~\ref{subsec:complexity}, we present the theoretical analysis of our algorithm and show that it achieves a $\tilde \gO(p \epsilon^{-4-2/p})$ upper bound on SFO calls.

%, and show that previous algorithms \citep{kwon2023fully,chen2023near} can be viewed as using forward difference to approximate the hyper-gradient. 

\subsection{Hyper-Gradient Approximation via Finite Difference} \label{subsec:relation}

% The main goal of \citep{kwon2023fully} is to save the computation of Hessian inverse or Hessian-vector products in Eq.~(\ref{eq:hyper-grad}).
% \citet{kwon2023fully} derived the F${}^2$SA method based from the penalty reformulation , which leading to  

The core idea of F${}^2$SA \citep{kwon2023fully} is to solve the reformulated penalty problem~(\ref{eq:penalty-Kwon}) and use the gradient of the penalty function to approximate the true hyper-gradient. 
To make connections of F${}^2$SA to the finite difference method, let us introduce the extra notation $g_{\nu}$ as the perturbed lower-lever problem with $\vy_{\nu}^*(\vx)$ and $\ell_{\nu}(\vx)$ being its optimal solution and optimal value, respectively:
% In this subsection, we revisit their hyper-gradient approximation and 
% We start from the following relationship of the difference between perturbed lower-level problems and the original problem.
% We let 
\begin{align*}
   g_{\nu}(\vx,\vy) &:= \nu f(\vx,\vy)+ g(\vx,\vy), \\
   \vy_{\nu}^*(\vx) &:= \arg \min_{\vy \in \sR^{d_y}} g_{\nu}(\vx,\vy), \\
   \ell_\nu(\vx) &:= \min_{\vy \in \sR^{d_y}} g_{\nu}(\vx,\vy),
\end{align*}
{Then we have
$\frac{\partial}{\partial \nu} \ell_\nu(\vx) \vert_{\nu=0} = \lim_{\nu \rightarrow 0} \frac{\ell_\nu(\vx) - \ell_0(\vx))}{\nu} = \lim_{\nu \rightarrow 0} f(\vx,\vy_\nu^*(\vx)) + \frac{g(\vx,\vy_\nu^*(\vx)) - g(\vx,\vy^*(\vx))}{\nu}$.} In our notation, we can rewrite 
\citep[Lemma B.3]{chen2023near} as
$\frac{\partial}{\partial \nu} \ell_\nu(\vx) \vert_{\nu=0} = \varphi(\vx)$. Similarly, we can also rewrite \citep[Lemma 3.1]{kwon2023fully} as
% Since the constraint $\vy = \arg \min_{\vz \in \sR^{d_z} } g(\vx,\vz)$ is equivalent to requiring $g(\vx,\vy) \le \min_{\vz \in \sR^{d_y}} g(\vx,\vz)$, it can be shown  that   holds under Assumption \ref{asm:LLSC} and~\ref{asm:prior}. Furthermore,  showed that the partial derivatives with respect to $\vx$ and $\nu$ are commutable, which leads to
\begin{align} \label{eq:nu-x-commute}
    \frac{\partial^2}{\partial \nu \partial \vx} \ell_\nu(\vx) \vert_{\nu=0} =  \frac{\partial^2}{\partial \vx \partial \nu} \ell_\nu(\vx) \vert_{\nu=0} = \nabla \varphi(\vx).
\end{align}
Let $\nu = 1/\lambda$ in Eq. (\ref{eq:penalty-Kwon}).
Then the fully first-order hyper-gradient estimator \citep{kwon2023fully,chen2023near} is exactly using forward difference to approximate $\nabla \varphi(\vx)$, that is,
\begin{align} \label{eq:prior-hypergrad-est}
     \frac{\frac{\partial}{\partial \vx} \ell_{\nu}(\vx) - \frac{\partial}{\partial \vx} \ell_0(\vx)}{\nu} \approx \frac{\partial^2}{\partial \nu \partial \vx} \ell_{\nu}(\vx) \vert_{\nu=0} = \nabla \varphi(\vx).
\end{align}
However, the forward difference is not the only way to approximate a derivative. Essentially, it falls into a general class of $p$th-order finite difference \citep{atkinson2005finite} that 
can guarantee an $\gO(\nu^p)$ approximation error. We restate this known result (with generalization to vector-valued functions) in the following lemma and provide a self-contained proof in Appendix \ref{apx:proof-err-pth} for completeness.
{\begin{lem} \label{lem:error-pth} 
Assume the function $\psi: \sR\rightarrow \sR^d$ has $C$-Lipschitz continuous $p$th-order derivative. 
There exist coefficients $\{ \alpha_j\}$ such that
\begin{align*}
    \left \Vert \frac{1}{\nu} \sum_{j} \alpha_j \psi(j \nu) - \psi'(0) \right \Vert = \gO(C\nu^p).
\end{align*}
If $p$ is even, the indices run $j = -p/2,\cdots, p/2$. If $p$ is odd, they run $j =-(p-1)/2, \cdots, (p+1)/2$. Furthermore, all the coefficients satisfy $\vert j \alpha_j \vert \le 1$ for all $j \ne 0$ and $\vert \alpha_0 \vert \le \mathbb{I}[p \text{ is odd}]$.
% If $p$ is even, there exist $p$th-order central difference  $\{ \alpha_j\}_{j=-p/2}^{p/2}$ such that
% where $\alpha_0 = 0$ and $\alpha_j =-\alpha_{-j}$ for all $j = 1,\cdots, p/2$. If $p$ is odd, there exist $p$th-order forward difference coefficients~$\{\beta_j \}_{j=0}^p$ such that  
% \begin{align*}
%      \left \Vert \frac{1}{\nu} \sum_{j=0}^{p} \beta_j \psi(j \nu) - \psi'(0) \right \Vert = \gO(C\nu^p),
% \end{align*}
% \begin{align} \label{eq:error-pth-differ}
%     \left \Vert \frac{1}{\nu} \sum_{j=-p/2}^{p/2} \alpha_j \frac{\partial}{\partial \vx} \ell_{j \nu}(\vx) - \frac{\partial^2}{\partial \nu \partial \vx} \ell_\nu(\vx) \vert_{\nu=0} \right \Vert = \gO(C \nu^p).
% \end{align}
% If $p$ is odd, then there exist coefficients $\{ \beta_j\}_{j=0}^{p}$ such that 
% \begin{align} \label{eq:error-pth-differ-odd}
%     \left \Vert \frac{1}{\nu} \sum_{j=0}^{p} \beta_j \frac{\partial}{\partial \vx} \ell_{j \nu}(\vx) - \frac{\partial^2}{\partial \nu \partial \vx}  \ell_\nu(\vx) \vert_{\nu=0} \right \Vert = \gO(C \nu^p).
% \end{align}
\end{lem}}
The explicit formulas for $\alpha_j$ can be found in Appendix \ref{apx:proof-err-pth}.
When $p=1$, we have $\alpha_0= -1$, $\alpha_1 = 1$, and we obtain the 
forward difference estimator $\psi(\nu) - \psi(0) / \nu$;
When $p=2$ we have $\alpha_{-1}=-1/2, \alpha_1=1/2 $ and we obtain the 
central difference estimator $ (\psi(\nu) - \psi(-\nu)) / (2\nu)$. Lemma \ref{lem:error-pth} tells us that in general we can always construct a finite difference estimator $\gO(\nu^p)$ error with $p$ points for even $p$ or $p+1$ points for odd $p$ under the given smoothness conditions. Inspired by Lemma~\ref{lem:error-pth} and Eq. (\ref{eq:nu-x-commute}) that $\frac{\partial^2}{\partial \nu \partial \vx} \ell_\nu(\vx) \vert_{\nu=0} = \nabla \varphi(\vx)$, we propose a fully first-order estimator via a linear combination of $\frac{\partial}{\partial \vx} \ell_{j\nu}(\vx)$ to achieve $\gO(\nu^p)$ approximation error to $\nabla \varphi(\vx)$ given that $\frac{\partial^{p+1}}{\partial \nu^p \partial \vx} \ell_\nu(\vx)$ is Lipschitz continuous in $\nu$.
It further leads to Algorithm \ref{alg:F2SA-p} that will be formally introduced in the next subsection.

% \begin{rmk}
% A subtlety to use Lemma \ref{lem:error-pth} for hyper-gradient estimation is that it only applies to a unitary function while $\frac{\partial}{\partial \vx} \ell_{\nu}(\vx)$ is a vector-valued function in $\nu$. However, the approximation error still holds for the whole vector under the Euclidean norm if we apply the lemma on each dimension and note that the finite difference coefficients are the same for all dimensions.    
% \end{rmk}

% Motivated by this fact, we propose using the following hyper-gradient estimator:
% \begin{align} \label{eq:our-hypergrad-est}
%      \frac{\nabla \ell_{\nu}(\vx)- \nabla \ell_{-\nu}(\vx)}{2 \nu} \approx \frac{\partial}{\partial \nu}\nabla \ell_{\nu}(\vx) \vert_{\nu=0} =  \nabla \varphi(\vx). 
% \end{align}

%Let $\lambda = 1/\nu$. 
%Then our method is equivalent to solving the symmetric penalty problem (\ref{eq:penalty-our}) instead of the asymmetric penalty problem (\ref{eq:penalty-Kwon}) used in prior work \citep{kwon2023fully,chen2023near}.

\subsection{The Proposed Algorithm} \label{subsec:algo}

\begin{algorithm*}[htbp]  
\caption{F${}^2$SA-$p$ $(\vx_0,\vy_0)$, ~ even $p$} \label{alg:F2SA-p}
\begin{algorithmic}[1] 
\STATE $ \vy_0^j = \vy_0$, ~$\forall j \in \sN$ \\[0.3mm]
\STATE \textbf{for} $ t =0,1,\cdots,T-1 $ \\[0.3mm]
\STATE \quad \textbf{parallel for} $j = -p/2,-p/2+1, \cdots,p/2$ \\[0.3mm]
\STATE \quad \quad $ \vy_{t}^{j,0} =  \vy_{t}^j$ \\[0.3mm]
\STATE \quad \quad \textbf{for} $ k =0,1,\cdots,K-1$ \\[0.3mm]
\STATE \quad \quad \quad Sample random i.i.d indexes $\{ (\xi^y_j,\zeta^y_j) \}$. \\[0.3mm]
\STATE \quad \quad \quad $ \vy_{t}^{j,k+1} = \vy_{t}^{j,k} - \eta_y \left( j \nu F_y (\vx_t,\vy_{t}^{j,k}; \xi^y_j) +  G_{y}(\vx_t,\vy_{t}^{j,k}; \zeta^y_j) \right)$ \\[0.3mm]
\STATE \quad \quad  \textbf{end for} \\[0.3mm]
\STATE \quad \quad $\vy_{t+1}^j = \vy_{t}^{j,K}$ \\[0.3mm]
\STATE \quad \textbf{end parallel for} \\[0.3mm]
\STATE \quad Sample random i.i.d indexes $\{ ({\xi_i^x},{\zeta_i^x}) \}_{i=1}^S. $ \\[0.3mm]
\STATE \quad Let $\{ \alpha_j\}_{j=-p/2}^{p/2}$ be the $p$th-order finite difference coefficients defined in Lemma \ref{lem:error-pth}. \\[0.3mm]
\STATE \quad $ \Phi_t = \frac{1}{S} \sum_{i=1}^S \sum_{j=-p/2}^{p/2} \alpha_j \left(
j F_x(\vx_t,\vy_{t+1}^j;\xi^{x}_i) + \dfrac{G_x(\vx_t,\vy_{t+1}^j;\zeta^{x}_i)}{\nu} \right)$ \\[0.3mm]
\STATE \quad  $\vx_{t+1} = \vx_t -  \eta_x {\Phi_t} / {\Vert \Phi_t \Vert}$ \\[0.3mm]
% \STATE \quad  $\vx_{t+1} = \vx_t -  \eta_x \dfrac{\Phi_t}{\Vert \Phi_t \Vert}$ \\[1mm]
% \STATE \quad \textbf{Option II: } $x_{t+1}  = x_t - \eta \hat \nabla \fL_{\lambda}^*(x_t) / \Vert \hat \nabla \fL_{\lambda}^*(x_t) \Vert $ \quad // Normalized Gradient Descent 
\STATE \textbf{end for}
\end{algorithmic}
\end{algorithm*}

Due to space limitations, we only present Algorithm \ref{alg:F2SA-p} designed for even $p$ in the main text. The algorithm for odd $p$ can be designed similarly, and we defer the concrete algorithm to Appendix~\ref{apx:proof-thm-upper-pth}.

Algorithm \ref{alg:F2SA-p} follows the double-loop structure of F${}^2$SA \citep{chen2023near,kwon2024complexity} and changes the hyper-gradient estimator to the one introduced in the previous section. Now, we give a more detailed introduction to the procedures of the two loops of F${}^2$SA-$p$.
\begin{enumerate}
    \item In the outer loop, the algorithm first samples a mini-batch with size $S$ and uses Lemma~\ref{lem:error-pth} to construct $\Phi_t$ via the 
linear combination of~$\frac{\partial}{\partial \vx} \ell_{j \nu}(\vx_t)$ for $j = -p/2,\cdots,p/2$ every iteration. After obtaining  $\Phi_t$ as an approximation to 
$\nabla \varphi(\vx_t)$, the algorithm then performs a normalized gradient descent step $\vx_{t+1} = \vx_t - \eta_x \Phi_t / \Vert \Phi_t \Vert$ with total $T$ iterations.
\item In the inner loop, the algorithm returns 
an approximation to $\frac{\partial}{\partial \vx}  \ell_{j \nu}(\vx_t)$ for all $j = -p/2,\cdots,p/2$. Note that Danskin's theorem indicates $\frac{\partial}{\partial \vx} \ell_{j \nu} (\vx_t) = \frac{\partial}{\partial \vx} g_{j\nu}(\vx_t,\vy_{j \nu}^*(\vx_t))$. It suffices to approximate $\vy_{j \nu}^*(\vx_t)$ to sufficient accuracy, which is achieved by taking a $K$-step single-batch SGD subroutine with stepsize $\eta_y$ on each function $g_{j \nu}(\vx,\,\cdot\,)$.
\end{enumerate}
\begin{rmk}[Effect of normalized gradient step]
Compared to \citep{chen2023near,kwon2023fully}, the only modification we make to the outer loop is to change the gradient step to a normalized gradient step. The normalization can control the change of $\vy_{j\nu}^*(\vx_t)$ and make the analysis of inner loops easier. We believe that all our theoretical guarantees also hold for the standard gradient step via a more involved analysis.
\end{rmk}

\subsection{Complexity Analysis} \label{subsec:complexity}

This section contains the complexity analysis of Algorithm \ref{alg:F2SA-p}. We first derive the following lemma from the high-dimensional Faà di Bruno formula \citep{licht2024higher}.
%based on the following recipe: (1) Show the $p$th-order smooth problem in Definition \ref{dfn:highly-smooth-blo} implies $\nabla \ell_\nu(\vx)$ is $p$th-order smooth in perturbation $\nu$ around zero. (2) Invoke Lemma \ref{lem:error-pth} 
%to show that F${}^2$SA-$p$ achieves $\gO(\nu^p)$-error to approximate the hyper-gradient, which further leads to an improved upper bound for highly smooth problems. In the following, we prove step (1) with 

% The main idea is to show the following: (1) additional high-order smoothness on the problem guarantees the additional high-order smoothness of  (2) for high-order smooth functions, we can use the central difference formula to obtain a better error guarantee for hyper-gradient approximation; (3) an improved error guarantee further leads to an improved complexity of our algorithm.

% \begin{lem} \label{lem:p2-lipschitz}
% Let $\nu \in (-1/\kappa,1/\kappa)$.
% Under Assumption \ref{asm:LLSC}, \ref{asm:prior} and \ref{asm:addition},
% $\frac{\partial^2}{\partial \nu^2} \nabla \ell_{\nu}(\vx)$ is $\gO(\bar L \kappa^5)$-Lipschitz continuous in $\nu$. % , where
% % \begin{align*}
% %     C = 2L_1 C_2 + 2 L_2 C_1^2 + 2L_1 C_3 + D_1 C_1 + D_2 C_1^2+ 4L_2 C_1C_2 = \gO(\bar L \kappa^5).
% % \end{align*}
% \end{lem}

% We notice that a similar result is also given by \citet[Lemma C.5]{}
% [TODO remark that the bound is tighter than Chen et al., 2025].

% In the following lemma, we show that, for $p$th-order smooth problems, we have that $\frac{\partial^p}{\partial \nu^p} \nabla \ell_\nu(\vx)$ is Lipschitz continuous in $\nu$ around $\nu = 0$. 

\begin{lem} \label{lem:solu-Lip-pth}
Let $\nu \in (0,1/(2\kappa)]$.
For any instance in the $p$th-order smooth bilevel problem class $\gF^{\text{nc-sc}}(L_0,\cdots,L_{p+1},\mu, \Delta)$ as Definition \ref{dfn:highly-smooth-blo},
$\frac{\partial^{p+1}}{\partial \nu^p \partial \vx} \ell_\nu(\vx)$ is $\gO(\kappa^{2p+1} \bar L)$-Lipschitz continuous in $\nu$.
\end{lem}

Our result generalizes the prior result for $p=1$ \citep{kwon2023fully} to any $p \in \sN_+$ and also tightens the prior bounds for $p=2$ \citep{chen2023near} as we remark in the following.

\begin{rmk}[Tighter bounds for $p=2$]
Note that the
variables $\vx$ and $
\nu$ play equal roles in our analysis. Therefore, our result in $p=2$ essentially implies that $\frac{\partial^3}{\partial \nu \partial \vx^2} \ell_\nu(\vx)$ is $\gO(\kappa^5 \bar L)$-Lipschitz continuous in $\nu$ around zero, which tightens the $\gO(\kappa^6 \bar L)$ bound of Hessian convergence in \citep[Lemma 5.1a]{chen2023near} and is of independent interest. The main insight is to avoid the direct calculation of $\nabla^2 \varphi (\vx) = \frac{\partial^3}{\partial \nu \partial \vx^2} \ell_\nu(\vx) \vert_{\nu = 0} $ which involves third-order derivatives and makes the analysis more complex, but instead always to analyze it through the limiting point $\lim_{\nu \rightarrow 0} \frac{\partial^3}{\partial \nu \partial \vx^2}  \ell_\nu(\vx)$.
\end{rmk}

% The formal proof to the above lemma can be found in Appendix  

Recall Eq. (\ref{eq:nu-x-commute}) that $\frac{\partial^2}{\partial \nu \partial \vx} \ell_{\nu}(\vx) \vert_{\nu=0} =  \nabla \varphi(\vx)$. Then Lemma \ref{lem:solu-Lip-pth}, in conjunction with Lemma \ref{lem:error-pth}, indicates that 
the $p$th-order finite difference used in F${}^2$SA-$p$ guarantees an $\gO(\nu^p)$-approximation error to $\nabla \varphi(\vx)$, which always improves the $\gO(\nu)$-error guarantee of F${}^2$SA \citep{kwon2023fully,chen2023near} for any $p \ge 2$. This improved error guarantee 
means that we can set $\nu = \gO(\epsilon^{1/p})$ to obtain an $\gO(\epsilon)$-accurate hyper-gradient estimator to $\nabla \varphi(\vx)$, which further leads to the following improved complexity of our algorithm.

%It will give an $\gO(\nu^p)$-approximation to the hyper-gradient, as we show in the following lemma.
% Recall that $\frac{\partial}{\partial \nu}\nabla \ell_{\nu}(\vx) \vert_{\nu=0} =  \nabla \varphi(\vx)$ by Eq. (\ref{eq:our-hypergrad-est}). If we use central difference formula in variable $\nu$ to approximate $\nabla \varphi(\vx)$, we have the following guarantee of the approximation error.

% \begin{lem} \label{lem:error-central}
% If $\frac{\partial^2}{\partial \nu^2} \nabla \ell_{\nu}(\vx)$ is $C$-Lipschitz continuous in $\nu$, then the central difference formula in Eq. (\ref{eq:our-hypergrad-est}) gives an $\gO(C \nu^2)$-approximation to $\nabla \varphi(\vx)$.
% \end{lem}
% Lemma \ref{lem:error-central} shows that the central difference has $\gO(\nu^2)$-approximation error, which improves the $\gO(\nu)$-error given by forward / backward difference. 
% Motivated by Lemma \ref{lem:error-pth} above, we propose the F$^2$SA-$p$ that uses $p$th-order finite difference for hyper-gradient approximation in Algorithm \ref{alg:F2SA-p}.  
% Below, we formally present the theoretical guarantee for F${}^2$SA-$p$.

\begin{thm}[Main theorem] \label{thm:upper-pth}
For any instance in the $p$th-order smooth bilevel problem class $\gF^{\text{nc-sc}}(L_0,\cdots,L_{p+1},\mu, \Delta)$ as per Definition \ref{dfn:highly-smooth-blo},
set the hyper-parameters as
\begin{align} \label{eq:hyper-setting-pth}
\begin{split}
    &\nu \asymp \min \left\{ 
    \frac{R}{\kappa}, \left( \frac{\epsilon}{\bar L \kappa^{2p+1}} \right)^{1/p} \right\}, ~~\eta_x \asymp \frac{\epsilon}{L_1 \kappa^3}, ~~ \eta_y \asymp \frac{\nu^2 \epsilon^2}{L_1 \kappa \sigma^2}, 
    \\
    &\quad \quad S \asymp \frac{\sigma^2}{\nu^2 \epsilon^2}, 
    ~~K \asymp 
    \frac{\kappa^2 \sigma^2}{\nu^2 \epsilon^2} \log \left(
    \frac{R L_1 \kappa}{\nu \epsilon}
    \right), ~~T \asymp \frac{\Delta}{\eta_x \epsilon},
\end{split}
\end{align}
where $R = \Vert \vy_0 - \vy^*(\vx_0) \Vert$. Run Algorithm \ref{alg:F2SA-p} if $p$ is even or Algorithm \ref{alg:F2SA-p-odd} (in Appendix \ref{apx:proof-thm-upper-pth})  if $p$ is odd. Then we can provably find an $\epsilon$-stationary point of $\varphi(\vx)$ with the total SFO calls upper bounded by
\begin{align*}
    p T(S+K)  = \gO \left( \frac{p \Delta L_1 \bar L^{2/p} \sigma^2\kappa^{9+ 2/p}}{\epsilon^{4+2/p}} 
    \log \left( \frac{R L_1 \bar L \kappa}{\epsilon} \right)
    \right).
\end{align*}
% When $p$ is odd, we can run Algorithm \ref{alg:F2SA-p-odd} with the same parameter setups to obtain the same guarantee.
\end{thm}

The above theorem shows that the F${}^2$SA-$p$ method can achieve the $\tilde \gO( p \kappa^{9+2/p} \epsilon^{-4-2/p} \log (\nicefrac{\kappa}{\epsilon}))$ SFO complexity for $p$th-order smooth bilevel problems. In the following, we give several remarks on the complexity in different regions of $p$.
\begin{rmk}[First-order smooth region]
For $p=1$, our upper bound becomes $\tilde \gO(\kappa^{11} \epsilon^{-6})$, which improves the  $\tilde \gO(\kappa^{12} \epsilon^{-6})$ bound in \citep{chen2023near} by a factor of $\kappa$. The improvement comes from a tighter analysis in the lower-level SGD update and a careful parameter setting.
\end{rmk}

\begin{rmk}[Highly-smooth region] \label{rmk:high-p-compl}
For $p = \Omega( \log (\nicefrac{\kappa}{\epsilon}) / \log \log (\nicefrac{\kappa}{\epsilon}) )$ in Definition \ref{dfn:highly-smooth-blo},
we can run F${}^2$SA-$q$ with $q \asymp \log (\nicefrac{\kappa}{\epsilon}) / \log \log (\nicefrac{\kappa}{\epsilon}) $ and the $\gO(q \kappa^9 \epsilon^{-4} (\nicefrac{\kappa}{\epsilon})^{2/q} \log(\nicefrac{\kappa}{\epsilon}))$ complexity in Theorem \ref{thm:upper-pth} simplifies to $\gO( \kappa^9 \epsilon^{-4} \log^3(\nicefrac{\kappa}{\epsilon}) / \log\log (\nicefrac{\kappa}{\epsilon})) = \tilde \gO(\kappa^9 \epsilon^{-4})$, which matches the best-known complexity for HVP-based methods \citep{ji2021bilevel} under stochastic Hessian assumption (\ref{eq:asm-stocBio}). 
\end{rmk}

In the upcoming section, we will derive an $\Omega(\epsilon^{-4})$ lower bound to prove that the F$^2$SA-$p$ is near-optimal in the above highly-smooth region if the condition number $\kappa$ is a constant. We leave the study of optimal complexity for non-constant $\kappa$ to future work. 

\paragraph{Comparison of results for odd $p$ and even $p$.}
Note that by Lemma \ref{lem:error-pth} when $p$ is odd, we need to use $p+1$ points to construct the estimator, which means the algorithm needs to solve $p+1$ lower-level problems in each iteration to achieve an $\gO(\nu^p)$ error guarantee.  In contrast, when $p$ is even, $p$ points are enough since the $p$th-order central difference estimator satisfies that $\alpha_0 = 0$. It suggests that even when $p$ is odd, the algorithm designed for odd $p$ may still be better. For instance, the F${}^2$SA-2 may always be a better choice than F${}^2$SA since its benefits \textit{almost come for free}: (1) it still only needs to solve 2 lower-level problems as the F${}^2$SA method, which means the per-iteration complexity remains the same. (2) Although the improved complexity of F${}^2$SA-2 relies on the second-order smooth condition, without such a condition, its error guarantee in hyper-gradient estimation will only degenerate to a first-order one, which means it is at least as good as F${}^2$SA.

\section{An $\Omega(\epsilon^{-4})$ Lower Bound}

In this section, we prove an $\Omega(\epsilon^{-4})$ lower bound for stochastic bilevel optimization via a reduction to single-level optimization.
Our lower bound holds for any randomized algorithms $\texttt{A}$, which can be defined as a sequence of measurable mappings $\{ \texttt{A}_t \}_{t=1}^T$ that is defined recursively by
\begin{align} \label{eq:dfn-alg}
    (\vx_{t+1},\vy_{t+1}) = \texttt{A}_t \left(r, F(\vx_0,\vy_0), G(\vx_0,\vy_0)), \cdots, F(\vx_t,\vy_t), G(\vx_t,\vy_t))\right ), ~~ t \in \sN_+,
\end{align}
where $r$ is a random seed drawn at the beginning to produce the queries, and $F,G$ are the stochastic gradient estimators that satisfy Assumption \ref{asm:stoc}.
Without loss of generality, we assume that $(\vx_0,\vy_0) = (\vzero,\vzero)$. Otherwise, we can prove the same lower bound by shifting the functions.

\paragraph{The construction.}
We construct a separable bilevel instance such that the upper-level function $f(\vx,\vy) \equiv f_\mU(\vx)$ and its stochastic gradient align with the hard instance in \citep{arjevani2023lower}, while the lower-level function is the simple quadratic $g(\vx,y) \equiv g(y) = \mu y^2 /2$ with deterministic gradients. We defer the concrete construction to Appendix \ref{apx:proof-lower}.
For this separable bilevel instance,
we can show that for any randomized algorithm defined in Eq. (\ref{eq:dfn-alg}) that uses oracles $(F_\mU, G)$, the progress in $\vx$ can be simulated by another randomized algorithm that only uses $F_\mU$, meaning that the single-level lower bound \citep{arjevani2023lower} also holds. 

\begin{thm}[Lower bound] \label{thm:lower}
There exist numerical constants $c>0$ such that for all $\Delta>0$, $L_1, \cdots,L_{p+1} >0$ and $\epsilon \le c \sqrt{L_1 \Delta}$,
there exists
a distribution over the function class $\gF^{\text{nc-sc}}(L_0,\cdots,L_{p+1},\mu,\Delta)$ and the stochastic gradient estimators satisfying Assumption \ref{asm:stoc}, such that
any randomized algorithm $\texttt{A}$ defined as Eq. (\ref{eq:dfn-alg}) can not find an $\epsilon$-stationary point of $\varphi(\vx) = f(\vx,\vy^*(\vx))$ in less than $\Omega(\Delta L_1 \sigma^2 \epsilon^{-4})$ SFO calls.
\end{thm}

% The analysis is simple using our fully separable
% construction $f(\vx,y) = f_\mU(\vx)$ and $g(y) = \mu y^2/2$. But we are a bit surprised that our straightforward construction is not used in prior works such as \citep{dagreou2024lower}. 

Below, we give a detailed discussion on the constructions in related works.
\paragraph{Comparison to other bilevel lower bounds.}
\citet{dagreou2024lower} proved lower bounds for finite-sum bilevel optimization via a similar reduction to single-level optimization.
However, the direct extension of 
their construction in the fully stochastic setting gives $f(\vx,\vy) = f_\mU(\vy)$ and $g(\vx,\vy) = (\vx- \vy)^2$, where the high-order derivatives of $f(\vx,\vy)$ not $\gO(1)$-Lipschitz in $\vy$ and thus violates our assumptions.
\citet{kwon2024complexity} also proved an $\Omega(\epsilon^{-4})$ lower bound for stochastic bilevel optimization. However, their construction $f(\vx,y) = y$ and $g(\vx,y) = (f_\mU(\vx) - y)^2$ violate the first-order smoothness of $g(\vx,y)$ in $\vx$ when $y$ is far way from $f_\mU(\vx)$. 
In this work, we use a fully separable construction to avoid all the aforementioned issues in other works.
% \paragraph{Comparison to existing lower bounds.} 

\section{Experiments}

In this section, we conduct numerical experiments to verify our theory.
Following \citep{grazzi2020iteration,ji2021bilevel}, we consider the ``learn-to-regularize'' problem of logistic regression (Example~\ref{exmp:l2reg}) on the ``20 Newsgroup'' dataset, which provably satisfies the highly smooth assumption of any order.
The dataset contains 18,000 samples, each sample consists of a feature vector in dimensional $130,107$ vector and a label that takes a value in $\{ 1,\cdots,20\}$. 
{We compare our proposed method F${}^2$SA-$p$ with $p \in \{ 2,3,5,8,10\}$} with both the previous best fully first-order method F${}^2$SA \citep{kwon2023fully,chen2023near} and other Hessian-vector-product-based methods stocBiO \citep{ji2021bilevel}, MRBO and VRBO \citep{yang2021provably}. We also include a baseline ``w/o Reg'' that means the training result of SGD without tuning any regularization. 
For all the algorithms, we 
search the other hyperparameters (including $\eta_x, \eta_y, \nu $) in a logarithmic scale with base $10$.
We run the algorithms with $K=10$ iterations in the inner loop, and $T = 1000$ iterations in the outer loop, and report the test loss/accuracy \textit{v.s.} the number of outer-loop iterations $t$ in Figure~\ref{fig:l2reg}.
{To demonstrate the potential of our methods on nonsmooth nonconvex problems, we also 
provide additional experiments on a 5-layer multilayer perceptron (MLP) network with ReLU activation in Appendix~\ref{apx:add-exp}.} The codes to reproduce our experiments are available online \footnote{https://github.com/TrueNobility303/F2BA}.

% We consider 

% We consider learning the regularization for two models, the first one is a linear model, and the second one is 

% that maps  features to $q=20$ classes.
% The whole dataset

% We tune $p$ in $\{1,2,\cdots,5\}$ and find that $p=2$ is the optimal choice. One possible reason is that the instance of both $p=1,2$ only requires solving two lower-level problems at each iteration, but the instance of $p \ge 3$ requires solving more than three lower-level problems and may not be concretely efficient. We regard F${}^2$SA-2 as 
% an important instantiation of F${}^2$SA-$p$ and present its concrete procedure as well as the comparison to F${}^2$SA in Appendix \ref{apx:alg-F2SA-2}. % We do not include F${}^2$SA-$p$ because it relies on higher-order smoothness assumptions and involves more hyper-parameters, making it not easy to tune in practice. 
% and present the experiment results in  where we also include 

% It can be observed that: (1) all the hessian-vector-product-based methods are worse than fully first-order methods; (2) the variance reduction technique in VRBO/MRBO is ineffective and may even harm the performance, which also aligns with the findings in \citep{liu2022bome}; (3) our method F${}^2$SA-2 significantly outperforms all the other algorithms. Our preliminary experiment results on the standard benchmark show the potential of F${}^2$SA-2 on large-scale bilevel problems. We provide ablation study of the parameter $p$ and additional experiments when $\vy$ is the parameters of in Appendix \ref{apx:add-exp}.

\begin{figure}[htbp]
    \centering
    \includegraphics[scale= 0.25]{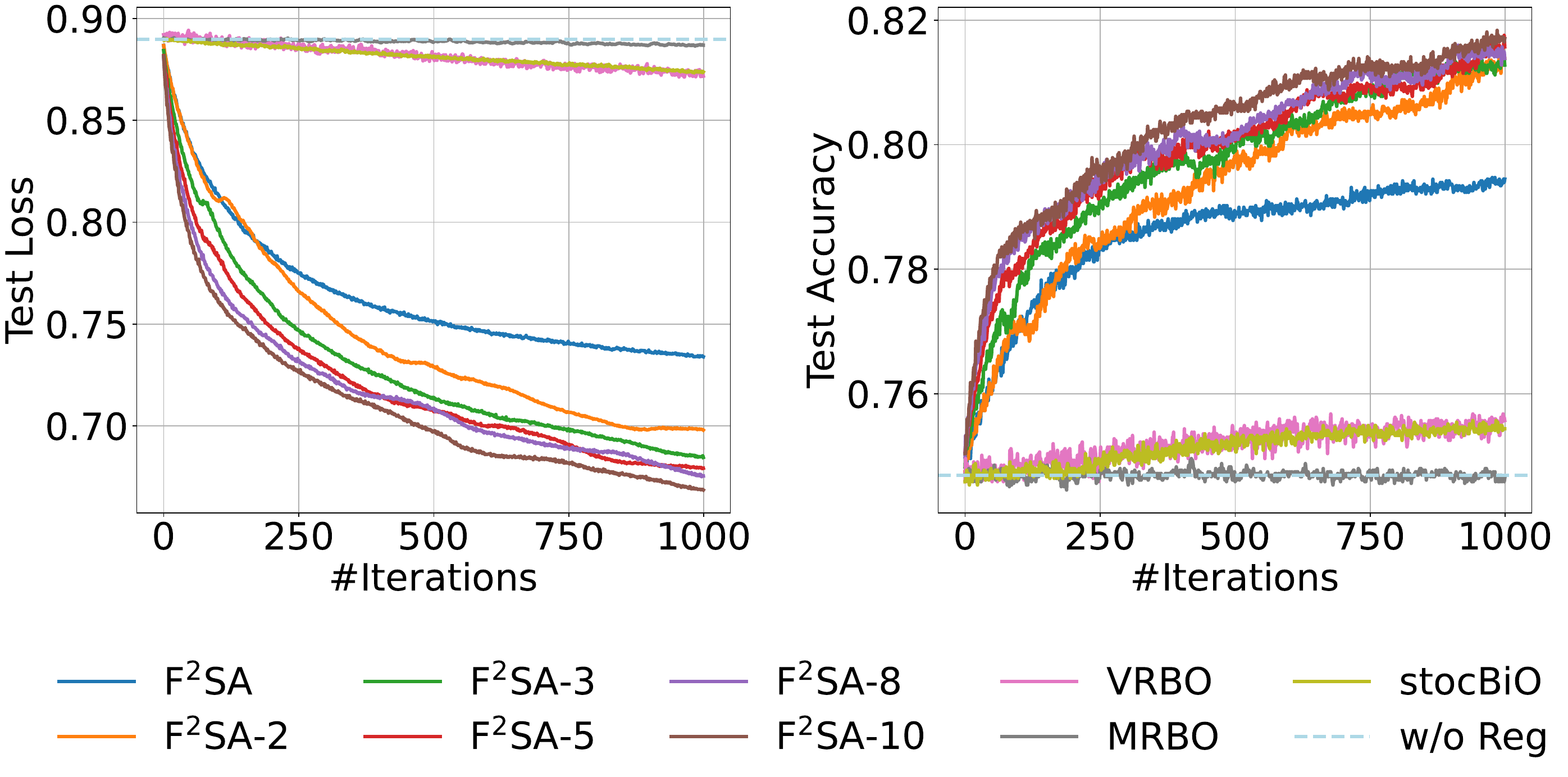}
\caption{Performances of different algorithms on Example \ref{exmp:l2reg}.}
\label{fig:l2reg}
\end{figure}

\section{Conclusions and Future Works}
This paper proposes a class of fully first-order method F${}^2$SA-$p$ that achieves the $\tilde \gO(p \epsilon^{-4-2/p})$ SFO complexity for $p$th-order smooth bilevel problems. Our result generalized the best-known $\tilde \gO(\eps^{-6})$ result  
\citep{kwon2024complexity,chen2023near} from $p=1$ to any $p \in \sN_+$.
We also complement our result with an $\Omega(\epsilon^{-4})$ lower bound to show that  our method is near-optimal when $p = \Omega(\log \epsilon^{-1} / \log \log \epsilon^{-1})$.
Nevertheless, a gap still exists when $p$ is small, and how to fill it even for the basic setting $p=1$ is an open problem.
Another open problem is whether our theory can be extended our theory to structured nonconvex-nonconvex bilevel problems studied by many recent works \citep{kwon2024penalty,chen2024finding,chen2025set,jiang2025beyond,xiao2023alternating,xiao2025unlocking}. {In addition, it will also be interesting to further improve the convergence rate of our methods by combining them with variance-reduction \citep{fang2018spider,cutkosky2019momentum} or momentum techniques \citep{fang2019sharp,cutkosky2020momentum}.}

\subsubsection*{Acknowledgments}
Lesi Chen thanks Jeongyeol Kwon for helpful discussions and Sanyou Mei for pointing out related references. Jingzhao Zhang is supported by Shanghai Qi Zhi Institute Innovation Program, Tsinghua Dushi Funds, and Xiongan AI Institute.


\bibliography{iclr2026_conference}
\bibliographystyle{iclr2026_conference}

\newpage
\appendix

\section{Notations for Tensors} \label{apx:notation-tensor}

We follow the notation of tensors used by \citet{kolda2009tensor}. For two $p$-way tensors $\gX \in \sR^{n_1 \times n_2 \times \cdots \times n_p}$ and $ \gY \in \sR^{n_1 \times n_2 \times \cdots \times n_p}$, their inner product $ z = \langle \gX, \gY \rangle  $ is defined as
\begin{align*}
\langle \gX, \gY \rangle = \sum_{i_1 =1}^{n_1} \sum_{i_2 = 1}^{n_2} \cdots \sum_{i_q = 1}^{n_p} \gX_{i_1,i_2,\cdots,i_p} \gY_{i_1,i_2,\cdots,i_p}.
\end{align*}
For two tensors $\gX \in \sR^{n_1 \times n_2 \times \cdots n_p}$ and $\gY \in \sR^{m_1 \times m_2 \cdots \times m_q}$,
their outer product $\gZ = \gX \otimes \gY$ is a tensor $\gZ \in \sR^{n_1 \times n_2 \times \cdots \times n_p \times m_1 \times m_2 \times \cdots \times m_q}$ whose elements are defined as
\begin{align*}
   \left( \gX \otimes \gY \right)_{i_1,i_2,\cdots,i_p,j_1,j_2,\cdots,j_q} = \gX_{i_1,i_2,\cdots,i_p} \gY_{i_1,i_2,\cdots,i_p}.
\end{align*}
% For a $\gX \in \sR^{n_1 \times n_2 \times \cdots n_p}$ and matrix $\mY^{m \times n_k}$, their mode-$k$ product $\gZ = \gX \bar \times_k\mY$ is a tensor $\gZ \in \sR^{n_1\times \cdots \times n_{k-1} \times m \times \cdots \times n_p }$ whose elements are defined as 
% \begin{align*}
%     \left( \gX \bar \times_k \mY \right)_{i_1\times \cdots \times i_{k-1} \times j \times \cdots \times i_p} = \sum_{i_k=1}^k \gX_{i_1,i_2,\cdots,i_p} \mY_{j, i_k}.
% \end{align*}
The operator norm of a tensor $\gX \in \sR^{n_1 \times n_2 \times \cdots \times n_p}$ is defined as
\begin{align*}
    \Vert \gX \Vert  = \sup_{ \Vert \vu_i \Vert = 1, i=1,\cdots,p } \langle \gX, \vu_1 \otimes \vu_2 \otimes \cdots \otimes \vu_p \rangle.
\end{align*}
Equipped with the notion of norm, we say a mapping $\gT: \sR \rightarrow \sR^{n_1 \times n_2 \times \cdots \times n_p}$ is $D$-bounded if
\begin{align*}
    \Vert \gT(\vx) \Vert \le D, \quad \forall \vx \in \sR.
\end{align*}
We say $\gT$ is $C$-Lipschitz continuous if 
\begin{align*}
    \Vert \gT(\vx) - \gT(\vy) \Vert \le C \Vert \vx - \vy \Vert, \quad \forall \vx,\vy \in \sR.
\end{align*}

\section{Proof of Lemma \ref{lem:error-pth}} \label{apx:proof-err-pth}
 
\begin{proof}
If $\psi^{(p)}(\nu)$ is $C$-Lipschitz continuous in $\nu$, then by Taylor’s theorem  we have
\begin{align} \label{eq:pth-Taylor}
    \psi(\nu) = \psi(0) + \sum_{k=1}^p \frac{(j \nu)^k}{k!} \psi^{(k)}(0) + \gO \left( C \nu^{p+1} \right).
\end{align}
Now, we analyze the case when $p$ is even or odd separately.
\paragraph{If $p$ is even.} 
The estimator we use is known as the $p$th-order central difference, whose coefficients are known \citep{khan2003mathematical}.
Let $n = p/2$.
We select coefficients $\{\alpha_j \}_{j=-n}^{n}$ such that
\begin{align*}
    \alpha_j = -\alpha_{-j}, \quad \forall j = 0,1,\cdots,n. 
\end{align*}
Then, summing up Eq. (\ref{eq:pth-Taylor}) with coefficients $\alpha_j$ gives 
\begin{align*}
    \frac{1}{\nu} \sum_{j=-n}^{j=n} \alpha_j \psi(j \nu) = \underbrace{2 \sum_{j=1}^{n} \alpha_j \sum_{k=1,3,\cdots}^{n-1} \frac{j^k \nu^{k-1}}{k!} \psi^{(k)}(0)}_{(*)} + \gO \left( C \nu^p \right).
\end{align*}
To let term (*) be equivalent to $\psi'(0)$, we let $\{ \alpha_j\}_{j=1}^{n}$ satisfy the following equations:
\begin{align*}
    2 \sum_{j=1}^{n} \alpha_j j^{k}  = \vone_{k=1}, \quad \forall k = 1,3,\cdots,n-1,
\end{align*}
which is equivalent to let $\{ j \alpha_j\}_{j=1}^{n}$ satisfy the following linear equation
\begin{align*}
    \begin{pmatrix}
    1 & 1 & 1 & \cdots & 1 \\
    1^2 & 2^2 & 3^2 & \cdots & n^{2} \\
    1^4 & 2^4 & 3^4 & \cdots & n^{4} \\
    \vdots & \vdots & \vdots & \ddots & \vdots \\
    1^{2(n-1)} & 2^{2(n-1)} & 3^{2(n-1)} & \cdots & n^{2(n-1)}
    \end{pmatrix}
    \begin{pmatrix}
    \alpha_1 \\
    2\alpha_2 \\
    3\alpha_3 \\
    \vdots \\
    n\alpha_{n}
    \end{pmatrix}
    = 
    \begin{pmatrix}
    1/2 \\
    0 \\
    0 \\
    \vdots \\
    0
    \end{pmatrix}.
\end{align*}
Now we solve this linear equation to determine the values of $\{\alpha_j\}_{j=1}^{n}$. Let $\mA$ be the coefficient matrix of this linear equation, and let $\mA_j$ be the matrix such that the $j$th column of $\mA$ is replaced by the standard unit vector $(1, 0, \cdots, 0)^\top$. By Cramer's rule, we have
\begin{align*}
    2 j \alpha_j = \frac{\det (\mA_j)}{\det(\mA)}, \quad j = 1,\cdots,n.
\end{align*}
By observation, we can find that both $\mA$ and $\mA_j$ are Vandermonde matrices. Therefore, we can explicitly calculate both $\det(\mA)$ and $\det (\mA_j)$ according to the determinant formula of  
Vandermonde matrices, which leads to
\begin{align*}
    2 j \alpha_j =& \frac{(-1)^{j-1}\cdot ((j-1)!)^2~ \cdot (n!)^2 \cdot j! \cdot (2j)!}{ (j!)^2 \cdot (j-1)! \cdot (n-j)! \cdot (2j-1)! \cdot (n+j)!} = \frac{2 (-1)^{j-1}  (n!)^2}{(n-j)! \cdot (n+j)!}.
\end{align*} 
Therefore, we have 
\begin{align*}
    \alpha_j = \frac{(-1)^{j-1}  (n!)^2}{j \cdot (n-j)! \cdot (n+j)!},
\end{align*}
from which it is clear that $\vert \alpha_j \vert \le 1/j$.
% Since the coefficient matrix
% in the above linear equation is a , we know this equation has a unique solution, which gives the 
\paragraph{If $p$ is odd.} {Instead of using the known $p$th-order forward difference \citep{khan2003mathematical} for which we find that the coefficients will be exponentially large in $p$, we motivate from the $p$th-order central difference above to obtain a stable estimator by leveraging negative points.}
Let $n = (p+1)/2$. We select 
coefficients $\{\alpha_j \}_{j=1-n}^{n}$ that satisfy the constraint $\sum_{j=1-n}^{n} \alpha_j = 0$. Then, 
summing up Eq.~(\ref{eq:pth-Taylor}) with coefficients $\alpha_j$ gives
\begin{align*}
    \frac{1}{\nu} \sum_{j=1-n}^{j=n} \alpha_j \psi(j \nu) = \underbrace{\sum_{j=1-n}^{j=n} \alpha_j \sum_{k=1}^{p} \frac{j^k \nu^{k-1}}{k!} \psi^{(k)}(0)}_{(*)} + \gO \left( C \nu^p \right).
\end{align*}
To let term (*) be equivalent to $\psi'(0)$, we let $\{ \alpha_j\}_{j=1-n}^{n}$ satisfy the following equations:
\begin{align*}
    \sum_{j=1-n}^{n} \alpha_j j^{k}  = \vone_{k=1}, \quad \forall k = 1,2,\cdots,p,
\end{align*}
which is equivalent to let $\{\hat \alpha_j \}_{j =(1-n), j \ne 0}^n$ satisfy the following linear equation
\begin{align*}
    \begin{pmatrix}
    1 & 1   &  \cdots  & 1 
    &  1 & \cdots & 1 \\
    1-n & 2-n   & \cdots & -1  & 1 & \cdots & n \\
    % (1-n)^2 & (2-n)^2 & (3-n)^2 & \cdots & (-1)^2  & 1^2 & \cdots & n^2 \\
    \vdots &  \vdots  & \vdots & \ddots & \vdots & \vdots & \ddots & \vdots \\
   (1-n)^{2n+1} & (2-n)^{2n+1}  & \cdots & (-1)^{2n+1}  & 1^{2n+1} & \cdots  & n^{2n+1}
    \end{pmatrix}
    \begin{pmatrix}
    \hat \alpha_{1-n} \\
    \hat \alpha_{2-n} \\
    \vdots \\
    \hat \alpha_n
    \end{pmatrix}
    = 
    \begin{pmatrix}
    1 \\
    0 \\
    \vdots \\
    0
    \end{pmatrix},
\end{align*}
where we denote $\hat \alpha_j = j \alpha_j$ for $j = 1-n,\cdots,-1$ and $1,\cdots,n$.
Now we solve this linear equation to determine the values of $\{\hat \alpha_j\}$. Let $\mA$ be the coefficient matrix of this linear equation, and let $\mA_j$ be the matrix such that the $j$th column of $\mA$ is replaced by the standard unit vector $(1, 0, \cdots, 0)^\top$. By Cramer's rule, we have
\begin{align*}
    \hat \alpha_j = \frac{\det (\mA_j)}{\det(\mA)}, \quad j = 1,\cdots,n.
\end{align*}
Similar to the case of even $p$, both $\mA$ and $\mA_j$ are Vandermonde matrices. Therefore, we can explicitly calculate both $\det(\mA)$ and $\det (\mA_j)$ according to the determinant formula of  
Vandermonde matrices. Then, for $j = 1, \cdots, n$, we can obtain that
\begin{align*}
    \alpha_j = \frac{\hat \alpha_j}{j} = \frac{(-1)^{j-1} (j-1)! (n-1)!  n! j!}{ j \cdot j! (j-1)! (n+j -1)! (n-j)!} = \frac{(-1)^{j-1} (n-1)! n!}{j (n+j-1)! (n-j)!}.
\end{align*}
Similarly, for $j = 1, \cdots, n-1$, we can obtain that
\begin{align*}
    \alpha_{-j} = \frac{\hat \alpha_{-j}}{-j} =  \frac{(-1)^{j} (n-1)! (j-1)! n!  j! }{j \cdot j ! (n+j-1)! (j-1)! (n+j)!} = \frac{(-1)^{j} (n-1)!  n!   }{j  (n-j-1)! (n+j)!}.
\end{align*}
Therefore, it is easy to see that $\vert \alpha_j \vert \le 1/j$ for $j = 1,\cdots, n$, and $\vert \alpha_{-j} \vert \le 1/j$ for $j = 1,\cdots,n-1$. Finally, we calculate $\alpha_0$ from the constraint $ \sum_{j=1-n}^n \alpha_j =0 $, which leads to 
\begin{align} \label{eq:alpha-0-odd}
    \alpha_0 = - \underbrace{\sum_{j=1}^n \frac{(-1)^{j-1} (n-1)! n!}{j (n+j-1)! (n-j)!}}_{S_1} - \underbrace{\sum_{j=1}^{n-1} \frac{(-1)^{j} (n-1)!  n!   }{j  (n-j-1)! (n+j)!}}_{S_2}.
\end{align}
We claim that $\alpha_0 = -1/n$ and hence $\vert \alpha_0 \vert \le 1$. We prove our claim by calculating the values of $S_1$ and $S_2$ to obtain $\alpha_0$.  For $S_1$, we have
\begin{align*}
    S_1 = \sum_{j=1}^n  (-1)^{j-1} \binom{n}{j}\frac{(n-1)! (j-1)!}{(n+j-1)!}.
\end{align*}
The fraction on the right is the Beta function $B(j,n)$, which can be represented as the integral $ 
B(j,n)= \int_0^1 x^{j-1} (1-x)^{n-1} {\rm d} x$. Therefore,
\begin{align*}
    S_1 =& \sum_{j=1}^n  (-1)^{j-1} \binom{n}{j} \int_0^1 x^{j-1} (1-x)^{n-1} {\rm d} x \\
    = & \int_0^1 (1-x)^{n-1} \left ( \sum_{j=1}^n (-1)^{j-1} \binom{n}{j} x^{j-1}  \right) {\rm d}x \\
    =& \int_0^1 \frac{(1-x)^{n-1}}{x} \left ( \sum_{j=1}^n (-1)^{j-1} \binom{n}{j} x^{j}  \right) {\rm d}x.
\end{align*}
Substituting the binomial expansion $(1-x)^n  = 1+ \sum_{j=0}^n \binom{n}{j} (-x)^j $, we then  have
\begin{align*}
     S_1 = \int_0^1 \frac{(1-x)^{n-1}}{x} \left( 1 - (1-x)^n \right) {\rm d} x.
\end{align*}
Let $y = 1-x$. We then have
\begin{align*}
    S_1 = \int_0^1 \frac{y^{n-1}}{1-y} (1- y^n) {\rm d} y.
\end{align*}
Substituting the geometric series sum $ \frac{1-y^n}{1-y} = \sum_{k=0}^{n-1} y^k  $, we then have
\begin{align*}
    S_1 =  \int_0^1  y^{n-1} \left(  \sum_{k=0}^{n-1} y^k \right) {\rm d} y = \int_0^1 \sum_{k=0}^{n-1} y^{n+k-1} {\rm d}y = \sum_{k=0}^{n-1}  \int_0^1 y^{n+k-1} {\rm d}y = \sum_{k=0}^{n-1} \frac{1}{n+k}.
\end{align*}
Following similar steps, we can also obtain that
\begin{align*}
    S_2 = -\sum_{k=0}^{n-2} \frac{1}{n+k+1}.
\end{align*}
Now, for $\alpha_0 = -S_1 - S_2$, the summation terms cancel out perfectly, which leads to $\alpha_0 = -1/n$.

\end{proof}

\section{Proof of Lemma \ref{lem:solu-Lip-pth}} \label{apx:proof-lem-solu-pth}

The proof relies on the high-dimensional version of the Faà di Bruno formula. To formally state the result, we define the following notions. For a mapping $\gT: \sR^m \rightarrow \sR^{n_1 \times \cdots \times n_q}$, we define its $k$th-order directional derivative evaluated at $\vz \in \sR^m$ along the direction $(\vu_1,\cdots,\vu_k)$ as   
\begin{align*}
    \nabla^k_{\vu_1,\cdots,\vu_k} \gT_{\mid \vz}= \nabla^k \gT_{\mid \vz} (\vu_1,\cdots,\vu_k).
\end{align*}
We let the symmetric products of $\vu_1,\cdots,\vu_k$ as 
\begin{align*}
    \vu_1 \lor \vu_2 \lor \cdots \lor \vu_k = \frac{1}{k!} \sum_{\pi \in {\rm Perm}(k)} \vu_{\pi(1)} \otimes \vu_{\pi(2)} \otimes \cdots \otimes \vu_{\pi(k)},
\end{align*}
where ${\rm Perm(k)}$ denotes the set of permutations of $\{1,2,\cdots,k \}$. Also, we define the set of all (unordered) partitions of a set $A$ into $k$ pairwise disjoint non-empty sets as
\begin{align*}
    \gP(A,k) = \left\{ \mP = (P_1,\cdots,P_k) \subseteq \gB(A) \mid A = \cup_{j=1}^k P_j;~ \emptyset \notin \mP; ~ P_i \cap P_j = \emptyset, \forall i<j  \right\},
\end{align*}
where $\gB(A)$ is the power set of $A$, \textit{i.e.}, the set of all subsets of $A$. We also abbreviate $\gP(\{ 1:q\},k)$ as $\gP(q,k)$.
Using the above notions, we have the following result.
\begin{lem}[{\citep[Proposition 3.1]{licht2024higher}}] \label{lem:faa-di-bruno}
Let $\gT_1$ and $\gT_2$ be two mappings. If $\gT_1$ and $\gT_2$ are $k$-times differentiable at the point $\vz$ and $\gT_1(\vz)$, respectively, then the composite mapping $\gT_2 \circ \gT_1$ is $k$-times differentiable at the point $\vz$ and we have 
\begin{align*}
    \nabla^q (\gT_2 & \circ \gT_1)_{\mid \vz} (\lor_{i=1}^q  \vu_i) = \sum_{\substack{1 \le k \le q,\\
    \mP \in \gP(q,k)}
    } \nabla^k {\gT_2}_{\mid \gT_1(\vz)} \left(\nabla^{\vert P_1 \vert} {\gT_1}_{\mid \vz}(\lor_{i \in P_1} \vu_i), \cdots \nabla^{\vert P_k \vert} {\gT_1}_{\mid \vz}(\lor_{i \in P_k} \vu_i)   \right).
\end{align*}
\end{lem}

Recall Danskin's theorem that $\frac{\partial}{\partial \vx} \ell_\nu(\vx) = \frac{\partial}{\partial \vx} g_\nu(\vx, \vy_\nu^*(\vx))$. We can apply Lemma \ref{lem:faa-di-bruno} with $\gT_1 = \vy_\nu^*(\vx)$ and $\gT_1 = \frac{\partial}{\partial \vx} g_\nu(\vx,\vy)$ to obtain that
\begin{align} \label{eq:high-order-x}
    \frac{\partial^{q+1}}{\partial \nu^q \partial \vx}  \ell_\nu(\vx)  &=  
    \sum_{\substack{1 \le k \le q,\\
    \mP \in \gP(q,k)}
    } \frac{\partial^{k+1}}{\partial \vy^k \partial \vx}g_\nu(\vx,\vy_\nu^*(\vx)) \left(
    \frac{\partial^{\vert P_1 \vert}}{\partial \nu^{\vert P_1 \vert}}
    \vy_\nu^*(\vx), \cdots, \frac{\partial^{\vert P_k \vert}}{\partial \nu^{\vert P_k \vert}}
    \vy_\nu^*(\vx)   \right).
\end{align}
Symmetrically, using the  first-order optimality condition $\frac{\partial}{\partial \vy} g_\nu(\vx,\vy_\nu^*(\vx)) = 0 $ and 
where the first identity uses the 
Lemma \ref{lem:faa-di-bruno} with $\gT_1 = \vy_\nu^*(\vx)$ and $\gT_1 = \frac{\partial}{\partial \vy} g_\nu(\vx,\vy)$ yields that
\begin{align}  \label{eq:high-order-y} 0  &=  
    \sum_{\substack{1 \le k \le q,\\
    \mP \in \gP(q,k)}
    } \frac{\partial^{k+1}}{\partial \vy^{k+1}} g_\nu(\vx,\vy_\nu^*(\vx)) \left(
    \frac{\partial^{\vert P_1 \vert}}{\partial \nu^{\vert P_1 \vert}}
    \vy_\nu^*(\vx), \cdots, \frac{\partial^{\vert P_k \vert}}{\partial \nu^{\vert P_k \vert}}
    \vy_\nu^*(\vx)   \right).
\end{align}
Since $\gP(q,1)$ contains only one element, the above identity implies that
\begin{align} 
\begin{split} \label{eq:recur-nabla-y-nu}
    \frac{\partial^{q}}{\partial \nu^q} \vy_\nu^*(\vx) &= -  \left(\nabla^2_{yy} g_{\nu}(\vx,\vy_\nu^*(\vx)) \right)^{-1} \sum_{\substack{2 \le k \le q,\\
    \mP \in \gP(q,k)}
    } \vw_{k, \mP}, \\
    {\rm where}~~ \vw_{k, \mP} &= \frac{\partial^{k+1}}{\partial \vy^{k+1}} g_\nu(\vx,\vy_\nu^*(\vx)) \left(
    \frac{\partial^{\vert P_1 \vert}}{\partial \nu^{\vert P_1 \vert}}
    \vy_\nu^*(\vx), \cdots, \frac{\partial^{\vert P_k \vert}}{\partial \nu^{\vert P_k \vert}}
    \vy_\nu^*(\vx)  \right). 
\end{split}
\end{align}
Based on Eq. (\ref{eq:recur-nabla-y-nu}), we can prove by induction that $\frac{\partial^q}{\partial \nu^q} \vy_\nu^*(\vx)$ is $\gO(\kappa^{2q+1})$-Lipschitz continuous in $\nu$ for all $q = 0,\cdots,p$. The induction base for $q = 0,1$ is already proved by \citet{chen2023near}.

\begin{lem}[{\citet[Lemma B.2 and B.5]{chen2023near}}] \label{lem:y-nu-lip}
Let $\nu \in (0,1/(2\kappa)]$.
Under Assumption \ref{asm:LLSC} and \ref{asm:prior},
$\vy_{\nu}^*(\vx)$ and $\frac{\partial}{\partial \nu} \vy_\nu^*(\vx)$ is $\gO(\kappa)$- and $\gO(\kappa^3)$-Lipschitz continuous in $\nu$, respectively.
\end{lem}

Since Eq. (\ref{eq:recur-nabla-y-nu}) also involves $(\nabla_{yy}^2 g_\nu (\vx,\vy_\nu^*(\vx)))^{-1}$, we also need the following lemma that gives its boundedness and Lipschitz continuity constants.

\begin{lem}[{\citet[Lemma B.1 and Eq. 18]{chen2023near}}] \label{lem:nabla-yy-lip}
Let $\nu \in (0,1/(2\kappa)]$.
Under Assumption \ref{asm:LLSC} and \ref{asm:prior},
$(\nabla_{yy} g_{\nu}(\vx,\vy_\nu^*(\vx)))^{-1}$ is $2/\mu$-bounded and $\gO(\kappa^2/\mu)$-Lipschitz continuous in $\nu $.
\end{lem}

In the remaining proofs, we will use Eq. (\ref{eq:recur-nabla-y-nu}) prove by induction that $\frac{\partial^q}{\partial \nu^q} \vy_\nu^*(\vx)$ is $\gO(\kappa^{2q+1})$-Lipschitz continuous in $\nu$, then we can easily use Eq. (\ref{eq:high-order-x}) to show that $\frac{\partial^{q+1}}{\partial \nu^q \partial \vx} \ell_\nu(\vx)$ is $\gO(\kappa^{2q+1} \bar L)$-Lipschitz continuous in $\nu$ for all $q=0,\cdots,p$. Note that the computational graph of either $\frac{\partial^q}{\partial \nu^q} \vy_\nu^*(\vx))$ or $\frac{\partial^{q+1}}{\partial \nu^q \partial \vx} \ell_\nu(\vx)$ in Eq. (\ref{eq:high-order-x}) or (\ref{eq:recur-nabla-y-nu}) defines a tree, where the root is output, the leaves are inputs, and the other nodes are the intermediate results in the computation. We can analyze the Lipschitz continuities of all the nodes from bottom to top using the following lemma.

\begin{lem}[{\citet[Lemma 12]{luo2022finding}}] \label{lem:composition}
Let $\gT_1 $ and $\gT_2$ be two tensor-to-tensor mappings.
If $\gT_1$ is $D_1$-bounded and $C_1$-Lipschitz continuous, $\gT_2$ is $D_2$-bounded and $C_2$-Lipschitz continuous, then the product mapping $\gT_1 \times \gT_2$ is $D_1D_2$-bounded and $(C_1 D_2 + C_2 D_1)$-Lipschitz continuous.
\end{lem}

\begin{proof}[Proof of Lemma \ref{lem:solu-Lip-pth}]
Now, we formally begin to prove by induction that $\frac{\partial^q}{\partial \nu^q} \vy_\nu^*(\vx)$ is $\gO(\kappa^{2q+1})$-Lipschitz continuous in $\nu$ for all $q = 0,\cdots,p$. Recall that the induction base follows Lemma \ref{lem:y-nu-lip}. In the following, we use the induction hypothesis that $\frac{\partial^k}{\partial \nu^k} \vy_\nu^*(\vx))$ is $\gO(\kappa^{2k+1})$-Lipschitz continuous in $\nu$ for all $k = 0,\cdots,q-1$ to prove that $\frac{\partial^{q}}{\partial \nu^{q}} \vy_\nu^*(\vx))$ is $\gO(\kappa^{2q+1})$-Lipschitz continuous in $\nu$. We know that $\frac{\partial^{k+1}}{\partial \vy^{k+1}} g_\nu(\vx,\vy_\nu^*(\vx))$ is $\gO(\bar L)$-bounded and $\gO(\kappa \bar L)$-Lipschitz continuous in $\nu$. Therefore, we can use Lemma \ref{lem:composition} to conclude that each $\vw_{k,\mP}$ is  $\gO(\kappa^{\sum_{j=1}^k (2 \vert P_j \vert -1)} \bar L ) = \gO(\kappa^{2q-k} \bar L)$-bounded and $\gO( \bar L \cdot \kappa^{2q-k+2}  + \kappa \bar L \cdot \kappa^{2q-k} ) = \gO(\kappa^{2q-k+2} \bar L)$-Lipschitz continuous in $\nu$. It further implies that the summation $\vw:=\sum_{\substack{2 \le k \le q,
    \mP \in \gP(q,k)}
    } \vw_{k, \mP}$ is $\gO(\kappa^{2q-2} \bar L)$-bounded and $ \gO(\kappa^{2q} \bar L) $-Lipschitz continuous in $\nu$. 
    Then, we can recall Lemma \ref{lem:nabla-yy-lip} that $(\nabla_{yy} g_{\nu}(\vx,\vy_\nu^*(\vx)))^{-1}$ is $2/\mu$-bounded and $\gO(\kappa^2/\mu)$-Lipschitz continuous in $\nu $, and use Eq. (\ref{eq:recur-nabla-y-nu}) to finish the induction  that $\frac{\partial^{q}}{\partial \nu^q} \vy_\nu^*(\vx) = -  \left(\nabla^2_{yy} g_{\nu}(\vx,\vy_\nu^*(\vx)) \right)^{-1} \vw$ is $\gO(\kappa^{2q+1})$-Lipschitz continuous in $\nu$ for all $q=0,\cdots,p$. Finally, by analogy with the similarity of Eq. (\ref{eq:high-order-x}) and (\ref{eq:recur-nabla-y-nu}), we can follow the same analysis to show that $\frac{\partial^{q+1}}{\partial \nu^q \partial \vx}  \ell_\nu(\vx)$ is  $\gO(\kappa^{2q+1} \bar L) $-Lipschitz continuous in $\nu$ for all $ = 0,\cdots,p$.  
\end{proof}

\section{Proof of Theorem \ref{thm:upper-pth}} \label{apx:proof-thm-upper-pth}

\begin{algorithm*}[htbp]  
\caption{F${}^2$SA-$p$ $(\vx_0,\vy_0)$, ~ odd $p$} \label{alg:F2SA-p-odd}
\begin{algorithmic}[1] 
\STATE $ \vy_0^j = \vy_0$, ~$\forall j \in \sN$ \\[1mm]
\STATE \textbf{for} $ t =0,1,\cdots,T-1 $ \\[1mm]
\STATE \quad \textbf{parallel for} $j = -(p-1)/2, \cdots,(p+1)/2$ \\[1mm]
\STATE \quad \quad $ \vy_{t}^{j,0} =  \vy_{t}^j$ \\[1mm]
\STATE \quad \quad \textbf{for} $ k =0,1,\cdots,K-1$ \\[1mm]
\STATE \quad \quad \quad Sample random i.i.d indexes $\{ (\xi^y_j,\zeta^y_j) \}$. \\[0.3mm]
\STATE \quad \quad \quad $ \vy_{t}^{j,k+1} = \vy_{t}^{j,k} - \eta_y \left( j \nu F_y (\vx_t,\vy_{t}^{j,k}; \xi_j^y) +  G_{y}(\vx_t,\vy_{t}^{j,k}; \zeta_j^y) \right)$ \\[1mm]
\STATE \quad \quad  \textbf{end for} \\[1mm]
\STATE \quad \quad $\vy_{t+1}^j = \vy_{t}^{j,K}$ \\[1mm]
\STATE \quad \textbf{end parallel for} \\[1mm]
\STATE \quad Sample random i.i.d indexes  $\{ ({\xi_i^x},{\zeta_i^x}) \}_{i=1}^S. $ \\[1mm]
\STATE \quad Let $\{ \alpha_j\}_{j=-(p-1)/2}^{(p+1)/2}$ be the $p$th-order finite difference coefficients defined in Lemma \ref{lem:error-pth}. \\[1mm]
\STATE \quad $ \Phi_t = \frac{1}{S} \sum_{i=1}^S \sum_{j=-(p-1)/2}^{(p+1)/2} \alpha_j \left(
j F_x(\vx_t,\vy_{t+1}^j;\xi^{x}_i) + \dfrac{G_x(\vx_t,\vy_{t+1}^j;\zeta^{x}_i)}{\nu} \right)$ \\[1mm]
\STATE \quad  $\vx_{t+1} = \vx_t -  \eta_x {\Phi_t} / {\Vert \Phi_t \Vert}$ \\[1mm]
% \STATE \quad  $\vx_{t+1} = \vx_t -  \eta_x \dfrac{\Phi_t}{\Vert \Phi_t \Vert}$ \\[1mm]
% \STATE \quad \textbf{Option II: } $x_{t+1}  = x_t - \eta \hat \nabla \fL_{\lambda}^*(x_t) / \Vert \hat \nabla \fL_{\lambda}^*(x_t) \Vert $ \quad // Normalized Gradient Descent 
\STATE \textbf{end for} \\[1mm]
\end{algorithmic}
\end{algorithm*}

In the main text, we only present the algorithm when  $p$ is even.
The algorithm when $p$ is odd follows a similar design, which is presented in Algorithm \ref{alg:F2SA-p-odd} for completeness. Our algorithms consist of a double loop, where the outer loop performs normalized SGD (NSGD) and the inner loop performs SGD. Before we give the formal proof, we first recall the convergence result for (N)SGD.

\begin{lem}[{\citet[Lemma 2]{cutkosky2020momentum}}] \label{lem:NSGD}
Consider the NSGD update $ \vx_{t+1} = \vx_t - \eta F_t / \Vert F_t \Vert$ to optimize a function $f : \sR^d \rightarrow \sR$ with $L$-Lipschitz continuous gradients. We have
\begin{align*}
    \frac{1}{T} \sum_{t=0}^{T-1} \E \Vert \nabla f(\vx_t)) \Vert \le \frac{3 (f(\vx_0) - \inf_{\vx \in \sR^d} f(\vx)) }{\eta T} + \frac{3 L \eta}{2} + \frac{8}{T} \sum_{t=0}^{T-1} \E \Vert F_t - \nabla f(\vx_t) \Vert.
\end{align*}
\end{lem}

\begin{lem}[{\citet[Lemma C.1]{kwon2024complexity}}] \label{lem:SGD}
Consider the SGD update $ \vx_{t+1} = \vx_t - \eta F_t$ to optimize a $\mu$-strongly convex function $f : \sR^d \rightarrow \sR$ with $L$-Lipschitz continuous gradients. Let $\vx^* = \arg \min_{\vx \in \sR^d} f(\vx)$ be the unique minimizer to $f$. Suppose $F_t$ is an unbiased estimator to $\nabla f(\vx_t)$ with variance bounded by $\sigma^2$. Setting $\eta < 2/ (\mu+ L)$, we have 
\begin{align*}
    \E \Vert \vx_t - \vx^* \Vert^2 \le (1 - \mu \eta )^t \Vert \vx_0 - \vx^* \Vert^2 + \frac{\eta \sigma^2}{\mu}.
\end{align*}
\end{lem}

The following two lemmas are also useful in the analysis. 

\begin{lem}[{\citet[Lemma 4.1]{chen2023near}}] \label{lem:varphi-smooth}
Under Assumption \ref{asm:LLSC}, and \ref{asm:prior}, the hyper-objective $\varphi(\vx) = f(\vx,\vy^*(\vx))$ is differentiable and has $L_\varphi = \gO(\bar L \kappa^3)$-Lipschitz continuous gradients.
\end{lem}

\begin{lem}[{\citet[Lemma B.6]{chen2023near}}] \label{lem:y-star-lip-x}
Let $\nu \in (-1/\kappa,1/\kappa)$.
Under Assumption \ref{asm:LLSC}, and \ref{asm:prior}, the optimal (perturbed) lower-level solution mapping $\vy_\nu^*(\vx) = \arg \min_{\vy \in \sR^{d_y}} \ell_v(\vx,\vy)$ is $4 \kappa$-Lipschitz continuous in $\vx$.
\end{lem}

Now we are ready to prove Theorem \ref{thm:upper-pth}. 
%Because the analysis when $p$ is odd or even is the same, we only give the proof when $p$ is even.

\begin{proof}[Proof of Theorem \ref{thm:upper-pth}]
We separately consider the complexity for the outer loop and the inner loop.
\paragraph{Outer Loop.} 
According to Lemma \ref{lem:varphi-smooth}, the hyper-objective $\varphi(\vx)$ has $L_\varphi = \gO(\bar L \kappa^3)$-Lipschitz continuous gradients.
If we can guarantee the condition
\begin{align} \label{eq:cond-lower-grad}
   \E \Vert \Phi_t - \nabla \varphi(\vx_t) \Vert  \le \frac{\epsilon}{32}, \quad t = 0,\cdots,T-1,
\end{align}
then we can further set $\eta_x = \nicefrac{\epsilon}{6 L_\varphi}$ and apply Lemma \ref{lem:NSGD} to conclude that the algorithm can provably find an $\epsilon$-stationary point of $\varphi(\vx)$ in $T = \left \lceil \nicefrac{6 \Delta}{\epsilon \eta_x} \right \rceil = \gO( \Delta L_1 \kappa^3 \epsilon^{-2}  ) $ outer iterations.
% Then we can apply Lemma \ref{lem:NSGD} on $\varphi(\vx)$
%     The outer loop performs the normalized gradient descent in $\vx$. We know from \citep[Lemma 2]{cutkosky2020momentum} that if we can obtain an $\gO(\epsilon)$-approximation to $\nabla \varphi(\vx)$, then this procedure finds an $\epsilon$-stationary point of $\varphi(\vx)$ in $T = \gO(\Delta L_\varphi \epsilon^{-2})$ iterations, where $L_\varphi = \gO(\bar L \kappa^3)$ according to \citep[Lemma 4.1]{chen2023near}.
\paragraph{Inner Loop.} From the above analysis, the remaining goal is to show that the inner loop always returns $\Phi_t$ 
satisfying Eq. (\ref{eq:cond-lower-grad}), which requires $\E \Vert \Phi_t - \nabla \varphi(\vx_t) \Vert =  \gO(\epsilon)$ for all $t = 0,\cdots,T-1$. Note that the setting of mini-batch size $S = 
\Omega \left( \nicefrac{\sigma^2}{\nu^2 \epsilon^2} \right) $ ensures that
\begin{align*}
\begin{cases}
     \E \left \Vert \Phi_t - \sum_{j=-p/2}^{p/2} \alpha_j \left( j \nabla_x f(\vx_t,\vy_{t+1}^j) + \dfrac{\nabla_x g(\vx_t,\vy_{t+1}^j)}{\nu} \right) \right \Vert =\gO(\epsilon), &\quad p \text{ is even}; \\[5mm]
      \E \left \Vert \Phi_t - \sum_{j=-(p-1)/2}^{(p+1)/2} \alpha_j \left( j \nabla_x f(\vx_t,\vy_{t+1}^j) + \dfrac{\nabla_x g(\vx_t,\vy_{t+1}^j)}{\nu} \right) \right \Vert =\gO(\epsilon), & \quad p \text{ is odd}.
\end{cases}
\end{align*}
By Lemma \ref{lem:solu-Lip-pth} and Lemma \ref{lem:error-pth}, setting $\nu = \gO( (\nicefrac{\epsilon}{\bar L \kappa^{2p+1}})^{1/p})$ can 
ensure that
\begin{align*}
\begin{cases}
      \left \Vert  \nabla \varphi(\vx_t) - \sum_{j=-p/2}^{p/2} \alpha_j \left( j \nabla_x f(\vx_t,\vy_{j \nu}^*(\vx_{t})) + \dfrac{\nabla_x g(\vx_t,\vy_{j \nu}^*(\vx_t))}{\nu} \right) \right \Vert =\gO(\epsilon), &\quad p \text{ is even}; \\[5mm] 
       \left \Vert  \nabla \varphi(\vx_t) - \sum_{j=-(p-1)/2}^{(p+1)/2} \alpha_j \left( j \nabla_x f(\vx_t,\vy_{j \nu}^*(\vx_{t})) + \dfrac{\nabla_x g(\vx_t,\vy_{j \nu}^*(\vx_t))}{\nu} \right) \right \Vert =\gO(\epsilon), &\quad p \text{ is odd}.
\end{cases}
\end{align*}
Therefore, a sufficient condition of $\E \Vert \Phi_t - \nabla \varphi(\vx_t) \Vert =  \gO(\epsilon)$ is 
\begin{align} \label{eq:cond-lower}
\begin{split}
    \begin{cases}
    \Vert \vy_{t+1}^j -  \vy_{j \nu}^*(\vx_t) \Vert  = \gO(\nu \epsilon / L_1), ~~ \forall j = -p/2,\cdots,p/2, &\quad  p \text{ is even}; \\[2mm]
    \Vert \vy_{t+1}^j -  \vy_{j \nu}^*(\vx_t) \Vert  = \gO(\nu \epsilon / L_1), ~~ \forall j = -(p-1)/2,\cdots,(p+1)/2, &\quad  p \text{ is odd}.
\end{cases}
\end{split}
\end{align}
Our next goal is to show that our parameter setting fulfills Eq. (\ref{eq:cond-lower}).
Note that for $\nu = \gO(1/\kappa)$, the (perturbed) lower-level problem $g_{j \nu}(\vx,\vy)$ is $\Omega(\mu)$-strongly convex in $\vy$ and has $\gO(L_1)$-Lipschitz continuous gradients jointly in $(\vx,\vy)$.
Therefore, if we set $\eta_y \lesssim 1/L_1$, then we can apply Lemma \ref{lem:SGD} on the lower-level problem $g_{j \nu}(\vx,\vy)$ to conclude that for ant $j$, we have
\begin{align*} 
    \E \Vert \vy_{t+1} -  \vy_{j \nu}^*(\vx_t) \Vert^2 \le  (1- \mu \eta_y)^K \Vert \vy_t - \vy_{j \nu}^*(\vx_t) \Vert^2 + \gO(\eta_y \sigma^2 / \mu).
\end{align*}
Comparing it with Eq. (\ref{eq:cond-lower}), we can set $\eta_y = \gO( \nicefrac{\nu^2 \epsilon^2}{L_1 \kappa \sigma^2})$ to ensure that for ant $j$, we have
\begin{align*} 
    \E \Vert \vy_{t+1} -  \vy_{j \nu}^*(\vx_t) \Vert \le  (1- \mu \eta_y)^K \Vert \vy_t - \vy^*_{j \nu}(\vx_t) \Vert + \gO( \nu \epsilon / L_1 ).
\end{align*}
Further, we can use Lemma \ref{lem:y-star-lip-x} and the triangle inequality to obtain that for ant $j$, we have
\begin{align} \label{eq:lower-sgd-recursion}
    \E \Vert \vy_{t+1} -  \vy_{j \nu}^*(\vx_t) \Vert \le  (1- \mu \eta_y)^K ( \Vert \vy_t - \vy_{j \nu}^*(\vx_{t-1}) \Vert + 4 \kappa \Vert \vx_{t} - \vx_{t-1} \Vert   )  + \gO( \nu \epsilon / L_1 ).
\end{align}
The recursion (\ref{eq:lower-sgd-recursion}) implies our setting of $K$ can ensure that Eq. (\ref{eq:cond-lower}) holds for all $t = 0,\cdots,T-1$.
We give an induction-based proof. To let the induction base holds for $t=1$, it suffices to set $K = \Omega( \nicefrac{\log (\nicefrac{R L_1}{\nu \epsilon})}{\mu \eta_y}) = \Omega( \nicefrac{\log (\nicefrac{R L_1}{\nu \epsilon}) \kappa^2 \sigma^2}{\nu^2 \epsilon^2})$, where
$\Vert \vy_{j \nu}^*(\vx_0) - \vy^*(\vx_0) \Vert^2 = \gO(R)$ is due to the setting of $\nu = \gO(\nicefrac{R}{ \kappa})$ and the fact that $\vy_\nu^*(\vx)$ is $\kappa$-Lipschitz in $\nu$ by Lemma \ref{lem:y-nu-lip}.
Next, assume that we have already guaranteed Eq.~(\ref{eq:cond-lower}) holds for iteration $t$, we prove that our setting of $K$ implies  Eq.~(\ref{eq:cond-lower}) holds for iteration $t+1$. Note that the NSGD update in $\vx$ means that $\Vert \vx_t  -\vx_{t-1} \Vert = \eta_x = \gO(\nicefrac{\epsilon}{6 L_1 \kappa^3})$. Therefore, Eq.~(\ref{eq:lower-sgd-recursion}) in conjunction with the induction hypothesis indicates that 
\begin{align*}
    \E \Vert \vy_{t+1} -  \vy_{j \nu}^*(\vx_t) \Vert \lesssim  (1- \mu \eta_y)^K \left( \frac{\nu \epsilon}{L_1} + \frac{\epsilon}{L_1 \kappa^2} \right)  +  \frac{\nu \epsilon}{L_1}.
\end{align*}
Therefore, we know that to let Eq. (\ref{eq:cond-lower}) holds for iteration $t+1$,
it suffices to let $K = \Omega( \nicefrac{\log(\nicefrac{1}{\nu \kappa^2})}{ \mu \eta_y} ) = \Omega( \nicefrac{\log (\nicefrac{1}{\nu \kappa^2}) \kappa^2 \sigma^2}{\nu^2 \epsilon^2})$. This finishes the induction.
\paragraph{Total Complexity.} According to the above analysis, we set $\nu \asymp  (\nicefrac{\epsilon}{\bar L \kappa^{2p+1}})^{1/p}$, $S \asymp \nicefrac{\sigma^2}{\nu^2 \epsilon^2}$,  $T \asymp \Delta L_1 \kappa^3 \epsilon^{-2}$, and $K \asymp \nicefrac{\log (\nicefrac{R L_1 \kappa}{\nu \epsilon}) \kappa^2 \sigma^2}{\nu^2 \epsilon^2} $ to ensure that the algorithm provably find an $\epsilon$-stationary point of $\varphi(\vx)$. Since $S \lesssim K$, the total complexity of the algorithm is
\begin{align*}
    &\quad p T (S +K) = \gO(p TK)=\gO \left( p \cdot \frac{\Delta L_1 \kappa^3}{\epsilon^2} \cdot \frac{\kappa^2 \sigma^2}{\nu^2 \epsilon^2} \log \left( \frac{R L_1 \kappa}{\nu \epsilon} \right) \right) \\
    &= \gO \left( \frac{p \Delta L_1 \bar L^{2/p} \sigma^2\kappa^{9+ 2/p}}{\epsilon^{4+2/p}} 
    \log \left( \frac{R L_1 \kappa}{\nu \epsilon} \right)
    \right).
\end{align*}
% To prove this claim, the theoretical guarantee of SGD \citep[Theorem D.2]{chen2023near} on the (perturbed) lower-level problems $g_{j \nu}(\vx,\vy)$ indicates that with $S_y = \Omega(\kappa^2 S_x)$ we have
% Below, we show how to prove Eq. (\ref{eq:cond-lower}) based on the above inequality. For the case $t=1$, we know that setting $K = \Omega(\kappa \log( \nicefrac{R L_1 \kappa}{\nu \epsilon}))$ can lead to Eq. (\ref{eq:cond-lower}), where $\Vert \vy_{j \nu}^*(\vx_0) - \vy^*(\vx_0) \Vert^2 = \gO(R)$ uses the setting of $\nu = \gO(\nicefrac{R}{p \kappa})$ and the fact that $\vy_\nu^*(\vx)$ is $\kappa$-Lipschitz in $\nu$ by Lemma \ref{lem:y-nu-lip}. Then we can prove all the cases $t \ge 2$ by induction. Suppose that Eq. (\ref{eq:cond-lower}) holds for $t=s-1$, then using the fact that $\vy_\nu^*(\vx)$ is $4\kappa$-Lipschitz continuous in $\vx$ \citep[Lemma B.6]{chen2023near} as well as the triangle inequality gives that
% \begin{align*}
%     \Vert \vy_{s} - \vy^*_{j \nu}(\vx_s) \Vert \le \Vert \vy_s - \vy^*_{j \nu}(\vx_{s-1}) \Vert + 4 \kappa \Vert \vx_{t+1} - \vx_t \Vert. 
% \end{align*}
% Now, recalling that we use a normalized gradient descent in $\vx$, which means $\Vert \vx_{t+1} - \vx_t \Vert = \eta_x$, we can further obtain that $\Vert \vy_s - \vy^*_{j \nu}(\vx_t) \Vert = \gO(\epsilon)$ by the induction base of the case $t=s-1$ as well as our setting of $\eta_x$. Then, according to Eq. (\ref{eq:lower-sgd}), we can set $K = \Omega( \kappa \log (\nicefrac{L_1 \kappa}{\nu} ))$ to let Eq. (\ref{eq:cond-lower}) hold for $t = s$ , which finishes the induction.
\end{proof}

\section{Proof of Theorem \ref{thm:lower}} \label{apx:proof-lower}

We prove our lower bound for stochastic nonconvex-strongly-convex bilevel optimization via a reduction to the lower bound for stochastic single-level nonconvex optimization \citep{arjevani2023lower}.
To state their lower bound, we first need to introduce the function class, oracle class, algorithm class, and the complexity measures.

\begin{dfn}\label{dfn:func-class-Arj}
Given any $\Delta>0$ and $L_1>0$, we use $\gF^{\rm nc}(L_1,\Delta)$ to denote the set of all smooth functions $f: \sR^d \rightarrow \sR$ that satisfies
\begin{enumerate}
    \item $f(\vzero) - \inf_{\vx \in \sR^d} f(\vx) \le \Delta$; 
    \item $\nabla f(\vx)$ is $L_1$-Lipschitz continuous.
\end{enumerate}
\end{dfn}

\begin{dfn}\label{dfn:oracle-class-Arj}
Given a function $\sR^d \rightarrow \sR$, we use $\sO(\sigma^2)$ to denote the set of all stochastic first-order oracles that return an unbiased stochastic estimator to $\nabla f$ with variance bounded by $\sigma^2$.
\end{dfn}

\begin{dfn} \label{dfn:alg-class-Arj}
Let $f: \sR^d \rightarrow \sR$ be a differentiable function and $F: \sR^d \rightarrow \sR$ be the stochastic estimator to $\nabla f$. 
A randomized first-order algorithm $\texttt{A}$ consists of a distribution $\gP_r$ over a measurable set $\gR$ and a sequence of measurable mappings $\{\texttt{A}_t \}_{t \in \sN}$ such that
\begin{align*}
    \vx_{t+1} = \texttt{A}_t(r, F(\vx_0), \cdots, F(\vx_t)), \quad t \in \sN_+,
\end{align*}
where $r \sim \gP_r$ is drawn a single time at the beginning of the protocol. We let $\gA_{\rm rand}$ to denote the class of all the algorithms that follow the above protocol.
\end{dfn}

\begin{dfn}\label{dfn:complexity-Arj}
We define distributional complexity of $\gA_{\rm rand}$ to find an $\epsilon$-stationary point of the functions  in $\gF^{\rm nc}(L_1,\Delta)$ with oracle $\sO(\sigma^2)$ as 
\begin{align*}
    {\rm Compl}_\epsilon(L_1,\Delta, \sigma^2) = \sup_{\texttt{O} \in \sO(\sigma^2)}\sup_{\gP_f \in \gP[\gF(\Delta,f)]}  \inf_{\texttt{A} \in \gA_{\rm rand}}  \inf \{ t \in \sN \mid \E \Vert \nabla f(\vx_t) \Vert \le \epsilon \},
\end{align*}
where the expectation is taken over the sampling of $f$ from $\gP_f$, the randomness in the oracle $\texttt{O}$, and the randomness in the algorithm $\texttt{A}$, $\{\vx_t \}_{t \in \sN}$ is the sequence generated by $\texttt{A}$ running on $f$ with oracle $\texttt{O}$, and $\gP[\gF^{\rm nc}(L_1,\Delta)]$ denotes the set of all distributions over $\gF^{\rm nc}(L_1,\Delta)$.
\end{dfn}

All the above definitions are merely restatements of \citep[Section 2]{arjevani2023lower}.
Although Definition \ref{dfn:complexity-Arj} uses the definition of distributional complexity, by Yao's minimax principle is also a lower bound for the worst-case complexity. 
Now, we recall the construction in \citep{arjevani2023lower} for proving the $\Omega(\epsilon^{-4})$ lower bound. Formally, we define the randomized function
\begin{align} \label{eq:hard-fU}
f_\mU(\vx) = \frac{L_1 \beta^2}{\bar L_1} f^{\rm nc} (\rho(\mU^\top  \vx / \beta)) + \frac{L_1 \lambda}{2 \bar L_1} \Vert \vx \Vert^2,    
\end{align}
where $\bar L_1 = 155$, $\beta = 4 \bar L_1 \epsilon / L_1 $, $\rho: \sR^T \rightarrow \sR^T$ is $\rho(\vx) = \vx \big / \sqrt{1 + \Vert \vx \Vert^2/ R^2}$, $R = 230\sqrt{T}$, $\lambda=1/5$, and
$f_T : \sR^T \rightarrow \sR$ is the nonconvex hard instance introduced by \citet{carmon2020lower}:
\begin{align*}
    f^{\rm nc}(x) : = -\Psi(1) \Psi(x_1) + \sum_{i=2}^T [ \Phi(-x_{i-1}) \Phi(-x_i) - \Psi(x_{i-1}) \Phi(x_i)].
\end{align*}
In the above, the component functions $\Psi, \Phi:\sR \rightarrow \sR$ are defined as
\begin{align*}
    \Psi(t) = 
    \begin{cases}
        0, & t \le 1/2, \\
        \exp(1 - 1/(2t-1)^2) , & t< 1/2 
    \end{cases}
    \quad \text{and} \quad 
    \Phi(t) = \sqrt{\rm e} \int_{-\infty}^t \exp(-t^2/2) {\rm d} t.
\end{align*}
For the hard instance in Eq. (\ref{eq:hard-fU}), \citet{arjevani2023lower} further
defined the stochastic gradient estimator $F_\mU$ as 
\begin{align} \label{eq:stoc-grad}
F_\mU(\vx) = \frac{L_1 }{\bar L_1} \left( \beta (\nabla \rho(\vx))^\top \mU F_T(\mU^\top \rho(\vx)) + \lambda \vx \right).
\end{align}
In the above, $F_T:\sR^T \rightarrow \sR^T$ is the stochastic gradient estimator of 
$\nabla f^{\rm nc}$ defined by
\begin{align*} 
    [F_T(\vx)]_i = \nabla_i f^{\rm nc}(\vx) \left( 1 + \vone_{i > {\rm prog}_{1/4}(x)} (\xi/\gamma-1) \right), \quad \xi \sim {\rm Bernoulli}(\gamma),
\end{align*}
where ${\rm prog}_{\alpha}(x) = \max\{ i\ge 0 \mid \vert x_i \vert > \alpha \}$  and $\gamma =  \min\{ (46 \epsilon)^2 / \sigma^2,1 \}$.  
For the above construction,
\citet{arjevani2023lower} showed the following lower bound.

\begin{thm}[{\citep[Theorem 3]{arjevani2023lower}}] \label{thm:lower-Arj}
There exist numerical constants $c,c'>0$ such that for all $\Delta>0$, $L_1>0$ and $\epsilon \le c \sqrt{L_1 \Delta}$,
the construction of function $f_\mU: \sR^d \rightarrow \sR $ and stochastic first-order oracle $F_\mU: \sR^d \rightarrow \sR$  in Eq. (\ref{eq:hard-fU}) and (\ref{eq:stoc-grad}) together give a distribution over the function class $\gF^{\rm nc}(L_1,\Delta)$ and a stochastic first-order oracle $\texttt{O} \in \sO(\sigma^2)$ such that
    \begin{align*}
        {\rm Compl}_\epsilon(L_1,\Delta, \sigma^2) \ge  c' \Delta L_1 \sigma^2 \epsilon^{-4}. 
    \end{align*}
\end{thm}

\begin{proof}[Proof of Theorem \ref{thm:lower}]
For any randomized algorithm $\texttt{A}$ defined as Eq.~(\ref{eq:dfn-alg}) running it on our hard instance, we show that it can be simulated by another randomized algorithm running on the variable $\vx$ such that Theorem~\ref{thm:lower-Arj} can be applied. Since $G(y) = \mu y$ is a deterministic mapping we know that any randomized algorithm $\texttt{A}$ induces a sequence of measurable mappings $\{\texttt{A}_t'\}_{t \in \sN}$ such that
\begin{align*} 
(\vx_t, y_t) &= \texttt{A}_t' (\xi, F(\vx_0), \cdots, F(\vx_{t-1}), y_0, \cdots, y_{t-1} ).
\end{align*}
Expanding the recursion for $y_t$ shows that the above equation induces another sequence of measurable mappings
$\{\texttt{A}_t''\}_{t \in \sN}$ such that
\begin{align*}
    (\vx_{t}, y_t) = \texttt{A}_t'' (\xi, F(\vx_0), \cdots, F(\vx_{t-1})).
\end{align*}
Therefore, we can apply Theorem \ref{thm:lower-Arj} to complete the proof.
\end{proof}

% \section{Proof of Proposition \ref{prop:zr-det}}

% \propzr* 

% \begin{proof}
% Following \citep[Proposition 1]{carmon2020lower}, for every $(f,g) \in \gF$ with $f: \sR^{d_x} \rightarrow \sR$ and  $g: \sR^{d_y} \rightarrow \sR$ we construct an column orthogonal matrix $\mU \in \sR^{d \times (d+T)}$ with $\mU^\top \mU = \mI_d$ and $d = d_x+d_y$ such that 
% \begin{align} \label{eq:iter-eq}
%     (\vx^{t}, \vy^t) = \mU^\top (\va^t, \vb^t), \quad \forall i=1,\cdots,T,
% \end{align}
% where $\{(\va^t, \vb^t)\}_{t=0}^{\infty}$ and $\{(\vx^t,\vy^t) \}_{t=0}^{\infty}$ is the iterates produced by $\gA^{\rm det}$ operating on the rotated function pair $(f_{\mU}, g_{\mU})$ and  $\gA^{\rm zr}$ operating $(f,g)$, respectively, and $f_{\mU}(\vx) = f(\mU^\top \vx)$, $g_{\mU}(\vx) = g(\mU^\top \vx)$. If $(f,g)$ can give the $\Omega(T)$ lower bound for $\gA^{\rm zr}$ such that  $\Vert \nabla \varphi(\vx^t) \Vert  > \epsilon$ for all $t = 0,1,\cdots,T$, then Eq. (\ref{eq:iter-eq}) also implies the same $\Omega(T)$ lower bound for $\gA^{\rm det}$. It is because
% \begin{align*}
%     \Vert \nabla \varphi_{\mU}(\va^t) \Vert  &= \Vert (\nabla_x f_{\mU} - \nabla_{xy}^2 g_{\mU} [ \nabla_{yy}^2 g_{\mU}]^{-1} \nabla_y f_{\mU}))[\va^t,\vy^*(\va^t))] \Vert \\
%     &= 
% \end{align*}
% \end{proof}

% \section{Future Works} \label{apx:future}

% \newpage
\section{The F${}^2$SA-2 Algorithm}
\label{apx:alg-F2SA-2}

We present the realization of F${}^2$SA-$p$ when $p=2$ in Algorithm \ref{alg:F2SA-2} to further compare its procedure with the original F${}^2$SA algorithm. 
Let $\lambda = 1/\nu$. We can observe that F${}^2$SA \citep{kwon2023fully,chen2023near} solves the following \textit{asymmetric} penalty problem
\begin{align*}
    \min_{\vx \in \sR^{d_x}, \vy \in \sR^{d_y}} f(\vx,\vy) + \lambda \left(g(\vx,\vy) - \min_{\vz \in \sR^{d_y}} g(\vx,\vz) \right),
\end{align*}
As an improved method, the F${}^2$SA-2 solved the following \textit{symmetric} penalty problem:
\begin{align*}
    \min_{\vx \in \sR^{d_x}, \vy \in \sR^{d_y}} \frac{1}{2} \left( f(\vx,\vy) + \lambda f(\vx,\vy) - \min_{\vz \in \sR^{d_y}} (-f(\vx,\vz) + \lambda g(\vx,\vz))) \right).
\end{align*}
The latter is better since the symmetric form makes the first-order approximation error to $\nabla \varphi(\vx)$ perfectly cancel out and leave only the second-order error term. Therefore, in terms of the theoretical guarantee by Theorem \ref{thm:upper-pth}, the $\tilde \gO(\epsilon^{-5})$ upper bound of F${}^2$SA-2 can improve the $\tilde \gO(\epsilon^{-6})$ upper bound of F${}^2$SA by a factor of $\epsilon^{-1}$.

% We notice that a similar algorithm has also been proposed by \citep{chayti2024new} in meta-learning. But they only focused on the deterministic setting, this work studied stochastic problems.

\begin{algorithm*}[htbp]  
\caption{F${}^2$SA-2 $(\vx_0,\vy_0)$} \label{alg:F2SA-2}
\begin{algorithmic}[1] 
\STATE $ \vz_0 = \vy_0$ \\[1mm]
\STATE \textbf{for} $ t =0,1,\cdots,T-1 $ \\[1mm]
\STATE \quad $ \vy_t^0 =  \vy_{t}, ~ \vz_t^0 = \vz_{t}$ \\[1mm]
\STATE \quad \textbf{for} $ k =0,1,\cdots,K-1$ \\[1mm]
\STATE \quad \quad  Sample random i.i.d indexes $(\xi^y,\zeta^y) $ and $(\xi^z, \zeta^z)$. \\[1mm]
\STATE \quad \quad $ \vy_t^{k+1} = \vy_t^k - \eta_y \left( \nu F_y (\vx_t,\vy_t^k; \xi^y) +  G_{y}(\vx_t,\vy_t^k; \zeta^y) \right)$ \\[1mm]
\STATE \quad \quad $ \vz_t^{k+1} = \vz_t^{k}- \eta_y \left(-\nu F_y (\vx_t,\vz_t^t; \xi^z) +  G_{y}(\vx_t,\vz_t^k; \zeta^z) \right)   $ \\[1mm]
\STATE \quad \textbf{end for} \\[1mm]
\STATE \quad $\vy_{t+1} = \vy_t^{K},~ \vz_{t+1} = \vz_t^{K} $ \\[1mm]
\STATE \quad Sample random i.i.d indexes $\{ ({\xi_i^x},{\zeta_i^x}) \}_{i=1}^S. $ \\[1mm]
\STATE \quad {\small $ \Phi_t = \dfrac{1}{2} \sum_{i=1}^S \left(F_x(\vx_t,\vy_{t+1};{\xi_i^x}) + F_x(\vx_t,\vz_{t+1};\xi_i^x) +  \dfrac{G_x(\vx_t,\vy_{t+1};\zeta_i^x) - G_x(\vx_t,\vz_{t+1};\zeta_i^x) }{\nu} \right)$} \\[1mm]
\STATE \quad  $\vx_{t+1} = \vx_t -  \eta_x {\Phi_t} / {\Vert \Phi_t \Vert}$ \\[1mm]
% \STATE \quad \textbf{Option II: } $x_{t+1}  = x_t - \eta \hat \nabla \fL_{\lambda}^*(x_t) / \Vert \hat \nabla \fL_{\lambda}^*(x_t) \Vert $ \quad // Normalized Gradient Descent 
\STATE \textbf{end for} \\[1mm]
\end{algorithmic}
\end{algorithm*}
% \section{Additional Experiments on LLM Data-Cleaning} \label{apx:add-exp}

\newpage 

{\section{Additional Experiments}} \label{apx:add-exp}

This section provides additional experiments on finding the optimal per-parameter regularization of a 5-layer MLP with ReLU activation and the hidden layer size of $500$. Following the notation in Example \ref{exmp:l2reg}, we let $\vx \in \sR^d$ parameterize the regularization matrix via $\mW_\vx = {\rm diag} (\exp(\vx))$. 
We also let $\ell_{\rm val}$ and $\ell_{\rm tr}$ be the logistic loss of the network prediction on the validation set and training set, respectively. The problem to solve has the same formulation as Example \ref{exmp:l2reg}, as restated below:
 \begin{align} \label{eq:deep-l2reg}
    \min_{\vx \in \sR^{d}} \ell_{\rm val}(\vy), \quad {\rm s.t.} \quad \vy \in \arg \min_{\vy \in \sR^{d}} \ell_{\rm tr}(\vy) + \vy^\top \mW_{\vx} \vy.
\end{align}
The difference between Example \ref{exmp:l2reg} is that now the problem is nonsmooth nonconvex due to the use of the MLP model. We present the experiment results in Figure \ref{fig:deep-l2reg}.

\begin{figure}[htbp]
    \centering
    \includegraphics[scale= 0.25]{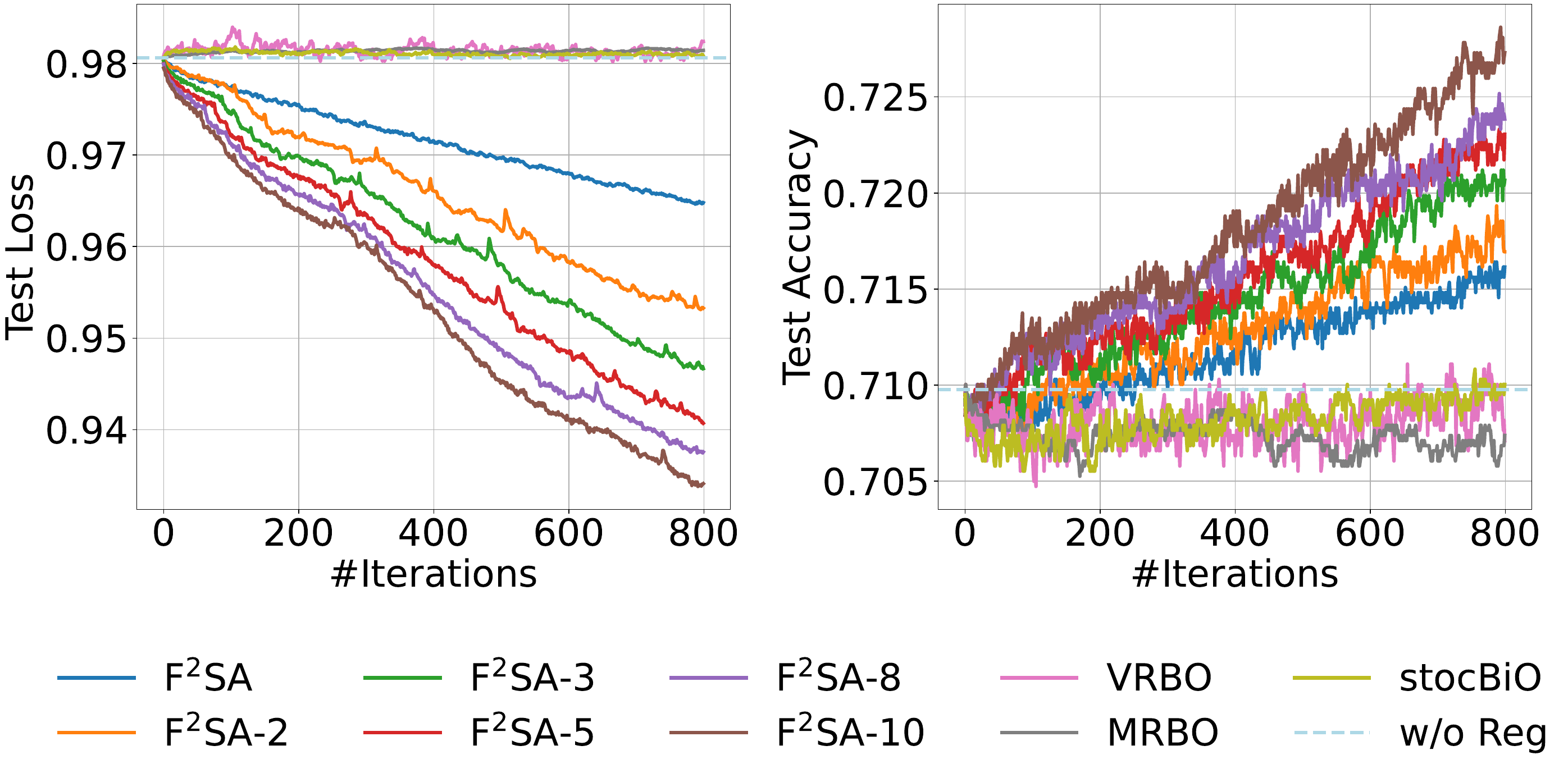}
\caption{Performances of different algorithms on Problem (\ref{eq:deep-l2reg}) with an MLP model.}
\label{fig:deep-l2reg}
\end{figure}

\section{Use of Large Language Models}

Large language models were used to help calculate the coefficient $\alpha_0$ when $p$ is odd in Lemma \ref{lem:error-pth}, and to refine wording and correct grammatical errors in parts of the paper.
% We provide additional experiments. 

% \begin{figure}[htbp]
%     \centering
%     \begin{tabular}{c c}
%     \includegraphics[scale=0.15]{fig/l2reg_stoc.pdf}  & \includegraphics[scale=0.15]{fig/l2reg_stoc_grad.pdf} \\
%     (a) Parallel setting  & (b) Non-parallel setting
%         \end{tabular}
% \caption{Performances of different algorithms when learning the optimal regularization under the parallel and non-parallel setting.}
% \label{fig:l2reg_stoc}
% \end{figure}

\end{document}